\documentclass[a4paper,reqno, 11pt]{article}
\usepackage{graphicx,color}
\usepackage{footmisc}
\usepackage{amsmath}
\usepackage{amsthm}
\usepackage{amsfonts}
\usepackage{amssymb}

\usepackage{biblatex}
\usepackage[utf8]{inputenc}
\numberwithin{equation}{section}
\usepackage[marginparwidth=3cm, marginparsep=.5cm]{geometry}
\usepackage{hyperref}
\hypersetup{
    colorlinks=true,
    linkcolor=blue,
    citecolor=red,
    urlcolor=blue,
    pdfborder={0 0 0}
}
\usepackage{cleveref}

\newcommand{\cR}{\mathcal R}

\newcommand{\cC}{\mathcal C}
\newcommand{\cM}{\mathcal M}
\newcommand{\cT}{\mathcal T}
\newcommand{\cH}{\mathcal H}
\newcommand{\cS}{\mathcal{S}}

\newcommand{\bbR}{\mathbb{R}}
\newcommand{\R}{\mathbb{R}}

\newcommand{\Z}{\mathbb{Z}}
\newcommand{\bbC}{\mathbb{C}}
\newcommand{\C}{\mathbb{C}}

\newcommand{\bbN}{\mathbb{N}}

\newcommand{\bbD}{\mathbb{D}}
\newcommand{\bbP}{\mathbb{P}}

\newcommand{\bbS}{\mathbb{S}}
\newcommand{\E}{\mathbb{E}}
\newcommand{\bbH}{\mathbb{H}}
\renewcommand{\P}{\mathbb{P}}
\newcommand{\cL}{\mathcal{L}}
\newcommand{\cF}{\mathcal{F}}

\newcommand{\norm}[1]{\lVert #1 \rVert}
\newcommand{\abs}[1]{\lvert #1 \rvert}

\DeclareMathOperator{\Tr}{Tr}
\DeclareMathOperator{\Var}{Var}

\renewcommand{\Re}{\operatorname{Re}}
\renewcommand{\Im}{\operatorname{Im}}

\newtheorem{theorem}{Theorem}[section]
\newtheorem{notation}[theorem]{Notation}
\newtheorem{remark}[theorem]{Remark}
\newtheorem{corollary}[theorem]{Corollary}

\newtheorem{lemma}[theorem]{Lemma}
\newtheorem{definition}[theorem]{Definition}
\newtheorem{proposition}[theorem]{Proposition}

  \crefname{theorem}{Theorem}{Theorems}
  \crefname{lemma}{Lemma}{Lemmas}
  \crefname{remark}{Remark}{Remarks}
  \crefname{proposition}{Proposition}{Propositions}
\crefname{notation}{Notation}{Notations}
\crefname{claim}{Claim}{Claims}
  \crefname{definition}{Definition}{Definitions}
  \crefname{corollary}{Corollary}{Corollaries}
  \crefname{section}{Section}{Sections}
  \crefname{figure}{Figure}{Figures}
    \crefname{assumption}{Assumption}{Assumptions}

\renewcommand{\d}{{\sharp \delta}}

\title{Central limit theorem for lozenge tilings with curved limit shape}

\author{Benoît Laslier\footnote{Suported by ANR DIMERS, number ANR-18-CE40-0033. Part of this project was carried while the author was on ``CNRS delegation'' and hosted in the university of Vienna.}\bigskip\\
\textit{Université de Paris, LPSM}\\
\url{laslier@lpsm.paris}}

\date{}

\begin{document}

\maketitle

\abstract{It has been well known for a long time that the height function of random lozenge tilings of large domains follow a law of large number and possible limits called dimer limit shapes are well understood. For the next order, it is expected that fluctuations behave like version of a Gaussian Free field, at least away from some special ``frozen'' regions. However despite being one of the main questions in the domain for 20 years, only special cases have been obtained. In this paper we show that for any specified limit shape with no frozen region, one can construct a sequence of domains whose height functions converge to that limit shape and where the height fluctuation converge to a variant of the Gaussian Free Field.}

\section{Introduction}

This paper studies large scale fluctuations for lozenge tiling with general limit shapes. Since both the model and the specific question of its fluctuations have a long history, we only present a short overview of the literature on the topic and refer to \cite{Gorin2020} for a more complete account. For us, the reason explaining the popularity of this model are two-fold, on the one hand it is a simple and physically relevant model to describe the statistical physics of surfaces, while on the other hand it is mathematically very rich. Before stating informally our result and presenting previous work on the model, let us describe the physical motivation and the heuristic derived from them.

\begin{figure}
	\begin{center}
		\includegraphics[width=.4\textwidth]{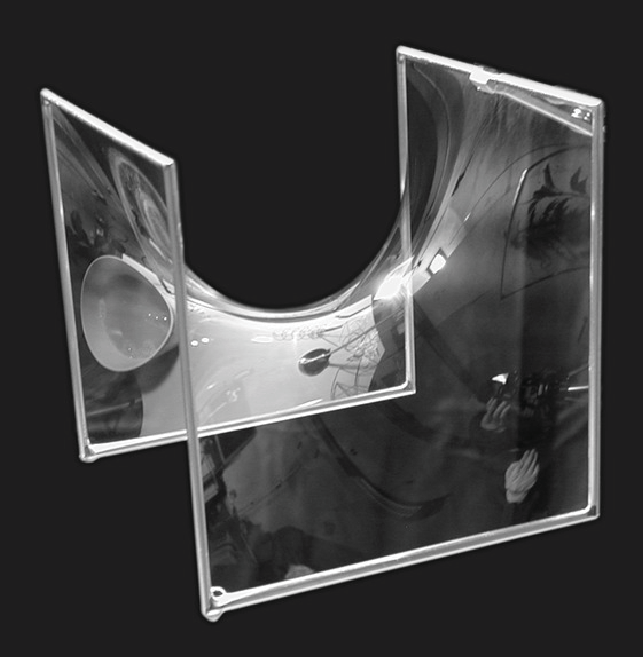}
	\end{center}
\end{figure}

We are interested here by physical surfaces appearing from statistical physics. The most obvious examples are of course soap bubbles whose shape result from the balance between a ``natural tendency'' of soap to retract and a constraint from the air inside (or from boundary conditions in the case of a soap film attached to a wire frame). Other and maybe more typical examples are surfaces defined as boundary of something: the interfaces between two domains with different orientations in a magnet, the boundary of an ice cube,$\ldots$. In all of these examples, there is a thermodynamic cost associated to the presence of a surface and we expect each physical system to find the configuration minimizing that cost. More precisely, we expect the surface $S$ to minimize (maybe subject to a constraint or boundary conditions) something like
\[
\int_{S} \sigma_{int}( \vec{n}_S (x) ) d^2 x
\]
where the integral is with respect to the area measure on the surface and $\sigma_{int}( \vec{n}_S (x) )$ is the cost associated to an infinitesimal portion of surface with normal vector $\vec{n}_S$, which is typically called surface tension of the system. While it is well known that for soap, $\sigma_{int}$ actually does not depend on the direction so $S$ is just a surface of minimal area, in general we expect a non trivial $\sigma_{int}$ : real life crystals for example tend to have macroscopic facets aligned with their underlying microscopic structure meaning that these directions are favoured in $\sigma_{int}$.
Note that the above expression already contains an implicit assumption that the behaviour of the model (or at least the surface tension) restricted to a small neighbourhood of a point only depends on the local normal vector to $S$ at that point.

To analyse fluctuations, it is useful to change the point of view and describe $S$ as the graph of a function $h$ (assuming that this can be done at least locally). In this new description the minimisation problem can be restated as
\[
\int_{S} \sigma_{int}( \vec{n}_S (x) ) d^2 x = \int \sigma (\nabla h(x) ) d^2x
\]
where $\sigma$ is essentially the same as $\sigma_{int}$ but with geometric factors included and a change of variable (in the following we will call $\sigma$ the surface tension even though it might be a slight abuse of language). Now if $h_0$ is the minimiser, we expect the fluctuations around $h_0$ to be small and to have a ``density'' in a sense given by
\[
\P (h_0 +\epsilon ) \propto \exp \int \sigma \big[ \nabla ( h_0 + \epsilon ) \big] - \int \sigma (\nabla h_0) \simeq \exp \int \sum_{i, j} \partial_{i,j} \sigma( \nabla h_0 ) \partial_{i,j} \epsilon,
\]
where the second equality is just a Taylor expansion using the fact that $h_0$ is a minimizer of the integral. When $\partial_{i,j} \sigma$ is constant, it is well known that we should interpret the right hand side as a saying that $\epsilon$ should be a Gaussian free field. When $\partial_{i,j} \sigma$ is not constant, we still expect a kind of Gaussian free field (abbreviated GFF in the following) but with a new geometry determined by $\partial_{i,j} \sigma (\nabla (h_0))$ (see \cite{Gorin2020} Chapter 12 for more details on this heuristic).

\bigskip

Why do lozenge tilings model such surfaces ? It is easy to see (for exemple in \cref{fig:exemple_tiling} a lozenge tiling as the image of a stack of cubes. Furthermore all possible tilings of a fixed domain clearly represent "stacks" with a common boundary so the model can be though of as analogous to a soap film attached to a wire frame.
 One of the great success of the mathematical study of lozenge tilings was the exact computation of $\sigma$ and the analysis of the solution $h_0$ to the associated minimisation problem (\cite{Cohn2001,Kenyon2007a}, see also the discussion below). In general, we believe that any minimiser $h_0$ will be analytical with non-extremal gradient in some part of its domain of definition called the liquid region and piecewise affine in the rest called the frozen region. In the latter, there should be \emph{no} large scale fluctuations (in fact not even microscopic ones which explain their name) while in liquid regions, fluctuations should be given by a kind of GFF as explained above. The purpose of this paper is to confirm the heuristic in the case where there is no frozen part in the domain. More precisely we will show that for any curve in space, if the associated limit shape has no frozen region, then we can find a sequence of discrete approximation of this curve such that the fluctuations of the associated lozenge tilings to a GFF in the geometry predicted by the heuristic.
 
 \begin{figure}
 	\begin{center}
 	\includegraphics[width=.8\textwidth]{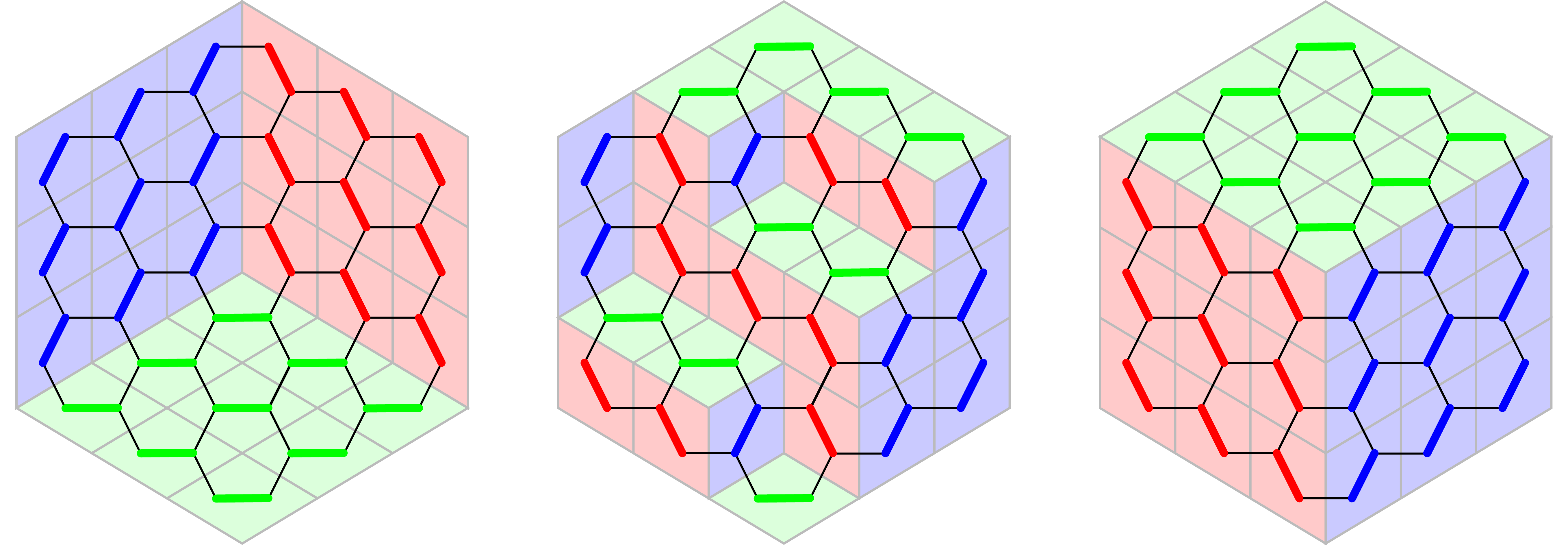}	
 	\caption{Three different lozenge tilings of the same region and the corresponding perfect matching or dimer cover of the hexagonal lattice.}\label{fig:exemple_tiling}
 	\end{center}
 \end{figure}

 \bigskip

There are obviously many ways to define a simple model for a random surface. As mentioned above, lozenge tilings are particularly appealing from a mathematical point of view for two related reasons: there are many different ways to view and analyse them (there are many ways in which the model is exactly solvable), and they have link to many fields of mathematics. Let us now discuss briefly the history of the model together with some of these connections.
Let us also note that it is often more convenient to describe a lozenge tiling by a kind of dual picture as a perfect matching of vertices in the hexagonal lattice (see \cref{fig:exemple_tiling}), which is also called dimer model on the hexagonal lattice. It turns out that many of the features of lozenge tilings can be generalized to dimer model on other planar bipartite graphs than the hexagonal lattice (including the square lattice where the model is equivalent to tilings with $2\times 1$ dominos). We will include in this discussion results obtained for dimers model more general that only lozenges.

A first very natural approach to understand the model is to follow the physical heuristics given above, but as of today this has been largely unsuccessful except for two very notable exceptions. Cohn-Kenyon-Propp \cite{Cohn2001} have proved the law of large number for the height function of large tilings exactly by interpreting $\sigma$ as the local energy-entropy cost of a piece of tiling, even though the computation of $\sigma$ itself uses different techniques which we will mention below. Also Sheffield \cite{Sheffield2005} identified all (non-extreme) translation invariant ergodic Gibbs measures (so conjecturally all possible local behaviour) of the model based mainly on the height function interpretation and ``finite energy'' type of results. In fact, it is in a sense partly because it is so hard to make sense directly of the physical intuition that the dimer model which allows for a variety of different method is so interesting.

Maybe the main technique was pioneered by Kasteleyn in \cite{Kasteleyn1961} and relies of the fact that the number of dimer covering of a (bipartite planar graph) can be written as $\det (K_G)$ where $K_G$ is a simple variant of the adjacency matrix of $G$. A corollary pointed out by Kenyon is that we can see individual dimers as a determinantal point process with kernel $K^{-1}_G$. Initially this was used on the torus and the whole plane for periodic graphs where the adjacency matrix can be essentially diagonalized in Fourrier space which is how Cohn-Kenyon-Propp computed $\sigma$ in \cite{Cohn2001}. This technique culminated in the analysis of all periodic full plane graphs in \cite{Kenyon2006} which in fact revealed and used a deep connection between the dimer model and algebraic geometry (through the rational integrals appearing in the expression of $\sigma$). This connection is still an active topic of research today (see \cite{Boutillier2020} and references there). 

A different set of methods more specific to lozenge and domino tilings relies on the combinatorics of partitions. Indeed a stack of cubes in a corner (which is described by a lozenge tiling with appropriate boundary condition) is a very natural generalisation with an extra dimension for an integer partition. In fact, to the best of our knowledge the first paper on lozenge tilings is \cite{MacMahon1915} and adopts this point of view to compute the number of tilings of an hexagon. The modern version of this approach relies on a ``line by line'' description of a tiling as a sequence of (interlaced) partitions with transition weights expressed in terms of Schur polynomial, see \cite{Okounkov2003, Gorin2016} for an early account of this version and an application to height fluctuations.

Still another idea establishes a link between dimers and discrete complex analysis. This first appeared in \cite{Kenyon2000} where Kenyon realised that on the square lattice, the equation $K_{\Z^2} K_{\Z^2}^{-1} = \text{Id}$ can be though of as a discrete version of the Cauchy identity. This enabled him to obtain the asymptotic behaviour of $K^{-1}$ on some domains of the square lattice (called Temperley domains) and to give the first proof of the convergence and conformal invariance of fluctuations. This was later generalised to isoradial graphs \cite{Kenyon2002, Tiliere2005} where it was also realised that dimers measures can have a locality property where all probabilities can be expressed only in terms of the local geometry of the (embedded) graph. This was particularly surprising because there are long range correlations in the model and is related to the notion of integrability in the sense of Yang-Baxter equation. We refer to \cite{Dubedat2014} for a particularly fine analysis using this method and to  \cite{Kenyon2018, Chelkak2020} for a more general approach in the same spirit.

Bijective combinatorics offers another angle on the model. Indeed there are bijections (or measure preserving transformations) from dimers to the 6-vertex model at the free fermion point, to two independent copies of the Ising model or to uniform spanning trees \cite{Dubedat2011a,Kenyon1}. Given the number of tools available to the study the dimer model, these bijections have mostly been used to go from dimers to these other models. However, this paper is part of a recent push to in the other direction from spanning tree to dimers \cite{Kenyon2007b,Kenyon2007, Laslier2013b, BLR16, BLRannex}. See \cref{sec:sketch} for a detailed discussing of the link between the current work and these references together with the overview of the proof.

This does not exhaust the list of approaches that have been used and several important results have been obtained with specific techniques. Aggarwal \cite{Aggarwal2019} obtained the local limit in arbitrary domains using a non-intersecting random walk representation related to the partition approach. Let us also mention that an interesting setup for non-intersecting random walk also appears at the boundary between phases with Airy processes offering a link to random matrix theory, see for example \cite{Johansson2003} for an early example. For fluctuations specifically, Petrov \cite{Petrov2012} used a double contour formula suitable for ¨asymptomatic to prove the CLT in many polygons ; Giuliani, Mastropietro Toninelli \cite{Giuliani2014} used a ferminic representation, constructive field theory theory and renormalisation to analyse interacting dimers.

\medskip

\subsection{Formal setup and statement of the theorem}\label{sec:setup}

In this paper we will denote the hexagonal lattice by $\cH$ and we will always consider it as being embedded in the plane using regular hexagons of side length $1$. We will write $\cH^\d := \delta \cH$ for $\delta > 0$ and more generally we will use the superscript $\d$ whenever we want to emphasize that something is a subset of $\cH^\d$ or a discrete object associated with $\cH^\d$.

When considering the height function of a dimer configuration on $\cH^\d$, unless specified otherwise we use the reference flow sending mass $1/3$ from white to black. In particular we emphasise that we do not scale the values of $h$ with $\delta$.

In the rest of the paper, we let $U$ denote an open bounded simply connected domain and we let $h^\cC$ be a continuous function from $U$ to $\R$ which extends continuously to $\partial U$. We always assume that $h^\cC$ is a dimer limit shape, i.e that  $h^\cC$ is the unique minimiser of
\[
\int_U \sigma ( \nabla h ), 
\]
with constraint $h = h^\cC$ on $\delta U$ and where $\sigma$ is the surface tension of uniform lozenge tilings. Finally we also assume that $h^\cC$ does not contain frozen points in the sense that $\nabla h^\cC$ is non extremal in $U$ (we allow it to become extremal on $\partial U$ however). Our main result concerns the convergence of the fluctuations of $h$ :

\begin{theorem}\label{thm:intro}
In the above setup, we can define subsets $U(\delta)$ of $\cH^\delta$ such that, writing $h^\d$ for the height function of a uniform dimer in $U(\delta)$ :
\begin{itemize}
\item Seen as closed set of $\C$, $U(\delta) \subset U$ and $U(\delta) \to U$.
\item $\sup_{z \in U(\delta)} |\delta h^\d( z) - h^\cC(z)| \to 0$ in probability.
\item There exists a map $\phi$ from $U$ to $\bbD$ such that
\[
h^\d - \E( h^\d) \to \frac{1}{2 \pi \chi} h^{\text{GFF}} \circ \phi
\]
where $h^{\text{GFF}}$ is a standard Dirichlet Gaussian free field in $\bbD$ and where the convergence is is law as a stochastic process. In the first statement, we identify $U(\delta)$ with the union of its closed bounded faces and the convergence is in Hausdorff sense.
\end{itemize}
\end{theorem}

Let us remark that by the law of large number from \cite{Cohn2001}, the second point just means that the boundary height of $U(\delta)$ approximates $h^\cC$ on $\partial U$.

Actually, it is necessary in the proof to ensure that all derivatives of $h^\cC$ and $\phi$ stay bounded so we will focus on the case where all our functions extend a bit beyond $U$. More precisely, we will prove the following proposition from which \cref{thm:intro} follows by a trivial diagonal argument.
\begin{proposition}\label{prop:main}
In the above context, assume further that $h^\cC$ can be extended to an open simply connected set $U_{ex}$ with such that $\bar U \subset U_{ex}$. Also assume that $h^\cC$ is still a limit height function in $U_{ex}$ and that it is $\cC^\infty$ and non extremal on $\bar U_{ex}$. Then we can define subsets $U(\delta)$ of $\cH^\d$ such that :
\begin{itemize}
\item Seen as closed set of $\C$ including all bounded faces, $U (\delta)$ is at distance $O(\delta)$ of $U$.
\item $\sup_{v \in \partial U(\delta)} |\delta h^{\d} (v) - h(v)| = O(\delta)$.
\item There exists a map $\phi$ from $U_{ex}$ to $\bbD$ such that
\[
h^\d - \E( h^\d) \to \frac{1}{2 \pi \chi} h^{\text{GFF}} \circ \phi
\]
where $h^{\text{GFF}}$ is a standard Dirichlet Gaussian free field in $\phi(U)$ and where the convergence is in law as a stochastic process. 
\end{itemize}
Furthermore, if we allow ourself to change the choice of $U_{ex}$ then, for the topology on $h^\cC$ given by the $L^3$ norm on all derivatives up to order $33$, the constants in the $O(\delta)$ terms and the map $\phi$ can be chosen as continuous functions of $h^\cC$. 
\end{proposition}
Note that in the second item, the way to assign a value of $h^\cC$ to each face of $\partial U(\delta)$ or the precise microscopic definition of this boundary (like internal or external boundary) is irrelevant since $h^\cC$ extends outside $U$ and is smooth.

Finally, let us note that while convergence is our main focus, the method gives relatively good ``qualitative'' estimates even for finite volume.
In particular the following two proposition control the distance between $h^\cC$ and the random discrete function $h^\d$ :
\begin{proposition}\label{prop:moment}
Suppose $U$ and $h^\cC$ are as in \cref{prop:main} and let $U(\delta)$ be the subsets of $\cH^\d$ defined in that proposition. For all $k\geq 2$ there exists $C_k$ such that, for all $\delta$ small enough, for all faces $f_1, \ldots, f_k$ of $U(\delta)$,
\[
\Big| \E [ \prod_{i=1}^k (h^\d (f_i) - \E[ h^\d( f_i) ]) ] \Big| \leq C_k( 1+ \log^{2k} r )
\]
where $r = \frac{1}{10} \min_{i \neq j} |f_i - f_j| \wedge \min_j d(f_j, \partial U)$. In particular taking $k$ faces at microscopic distance from each other we obtain
\[
\E [ |h^\d (f) - \E[ h^\d(f) ]|^k ] \leq C_k (1 + \log^{2k}\delta ).
\]
The constants $C_k$ can be chosen as a continuous functions of $h^\cC$ as in \cref{prop:main}.
\end{proposition}
For the expectation itself, even though we have $\E(\delta h^\d) - h^\cC = O(\delta)$ on the boundary, we cannot extend the bound inside $U$. We only have a much worse logarithmic control.
\begin{proposition}\label{prop:esperance}
Suppose $U$ and $h^\cC$ are as in \cref{prop:main} and let $U(\delta)$ be the subset of $\delta \cH$ defined by that proposition. There exists $C > 0$ such that for all face $f$ of $U(\delta)$,
\[
| \delta \E [h^\d(f) ] - h^\cC | \leq C \delta ( 1 + \log \frac{1}{\delta}). 
\]
The constant $C$ can be chosen as continuous functions of $h^\cC$ as in \cref{prop:main}.
\end{proposition}

\subsection{Sketch of proof and organisation of the paper}\label{sec:sketch}

As mentioned in the overview of previous results above, this paper is part of a series of work studying the dimer model using the link to uniform spanning tree (UST). To discuss the new input here and the structure of our proof, we need to first describe \cite{Kenyon2007b, Kenyon2007, Laslier2013b, BLR16} in more details.

On the the square lattice, there is a simple bijection between uniform dimers and UST, provided the boundary condition are well chosen. A fist extension of this bijection was found in \cite{Kenyon1} but it only applies to a specific class of graphs on the dimer side. The paper \cite{Kenyon2007b} introduced a further generalisation where from any finite dimer graph $G$, one construct an appropriate graph $\Gamma$ (called a T-graph) such that the (wired) UST on $\Gamma$ correspond almost bijectively to dimers on $G$. Unfortunately, the construction of $\Gamma$ provided in \cite{Kenyon2007b} is quite involved and the geometry of $\Gamma$ quite complex in general which so not much can be said about the UST of $\Gamma$.

 \cite{Kenyon2007} was the first attempt to use the construction of \cite{Kenyon2007b} to analyse fluctuations in the dimer model. First Kenyon was able to rephrase the analysis of fluctuations in the dimer model on $G$ as a natural question about discrete harmonic functions on $\Gamma$ (this is not so surprising given the close link between UST and random walk). Then he provided two methods to get around the issue that the geometry of $\Gamma$ is unknown ! In a fist part he solved explicitly the whole plane version of the construction of $\Gamma$ which allowed him to construct families of domains with an explicit associated T-graph but which had unfortunately only flat limit shapes. For domains with a curved limit shape, in a second part he found a way to build $\Gamma$ by starting from the desired limit shape, then building a crude approximation of $\Gamma$ by a discretisation procedure, and finally correcting possible defect ``by hand'' to build a good-enough auxiliary graph. Unfortunately, the analysis of discrete harmonic function (which has to be done up to quite small error) in the paper contains an error in a crucial estimate which makes the proof wrong and it is still unclear today whether the required estimate is even correct.

\cite{Laslier2013b} started as an attempt to correct the analysis of discrete harmonic functions from \cite{Kenyon2007} but only obtained a partial results : For the explicit whole plane T-graph mentioned in the previous paragraph,  the paper proved a CLT for the random walk. Unfortunately, while this does provide some control on harmonic function, it was not precise enough to prove the estimates required in \cite{Kenyon2007b}.

One of the main contribution of \cite{BLR16} (together with its companion paper \cite{BLRannex}) was to side-step the issue of analysing discrete harmonic functions on T-graphs. Using directly the UST, it was proved that convergence of the random walk on $\Gamma$ (and a few soft technical assumptions) is enough to deduce the convergence of dimer fluctuations. Together with \cite{Laslier2013b}, this completed the proof for the first part of \cite{Kenyon2007}.

\smallskip

The purpose of this paper is to complete the second part of \cite{Kenyon2007} and our overall strategy should be clear from the above discussion. Given a limit shape $h^\cC$, in \cref{sec:construction} we give the construction of the sequence of graphs $U(\delta)$ from \cref{prop:main} together with T-graphs $\Gamma(\delta)$. This section follows directly the idea from \cite{Kenyon2007} but has to be rewritten because we need to give more precise error terms and more details on the construction. \cref{sec:CLT} contains the CLT for the random walk on $\Gamma(\delta)$ and the other technical estimates necessary to apply \cite{BLR16}. We are not done because (due to the errors in the construction), there is no exact link between the spanning tree in $\Gamma(\delta)$ and the dimer model on a piece of the hexagonal lattice. Again following Kenyon's idea, we show in \cref{sec:comparison} that these two laws are still very close (even in total variation sense) to each other provided the initial approximation was precise enough, which concludes the proof. \cref{sec:Tgraph} contains a summary of known properties of T-graphs that will be needed for the paper.

\begin{remark}
The continuity with respect to $h^\cC$ in \cref{thm:intro} is in fact essentially analytical and essentially independent from the core of our arguments. In order to make the paper more readable, we will therefore write the whole paper for a fixed $h^\cC$ as if our goal was to prove \cref{thm:intro} without the continuity statement. We complete the analytical work to obtain the continuity in \cref{app:continuity}. 
\end{remark}

\paragraph{Acknowledgements} I thank Dmitry Chelkak, Richard Kenyon, Marianna Russkikh and Fabio Toninelli for interesting discussions that where useful for this paper. I am also extremely grateful to Nathanael Berestycki and Gourab Ray for countless exchanges, the whole approach of dimers through spanning trees could not have happened without them.

\section{Known results on T-graphs}\label{sec:Tgraph}

\subsection{Notations for the hexagonal lattice}

It is necessary for some computations to use explicit coordinates on the hexagonal lattice so we start by introducing our set of convention. We denote by $\cH$ the infinite hexagonal lattice and by $\cH^\dagger$ it's dual lattice, i.e. the infinite triangular lattice. We view $\cH$ as embedded in $\C$ using regular hexagons with edge length $\frac{\sqrt{3}}{3}$ as in \cref{fig:coordinate}. and call $e_\rightarrow := \frac{\sqrt{3}}{3}$, $e_\nwarrow := \frac{\sqrt{3}}{3}e^{2i\pi/3}$ and $e_\swarrow := \frac{\sqrt{3}}{3}e^{-2i\pi/3}$ the vectors in the direction of the edges of $\cH$.

We will use extensively the bipartite structure of $\cH$. We call vertices to the left of an horizontal edge white and vertices to the right black and typically use either $w$ or $b$ to denote vertices of $\cH$ depending on the colour.
To define coordinates, we fix a fundamental domain ${b_0, w_0}$ made of two vertices adjacent along an horizontal edge and set the coordinate of \emph{both vertices} on it as $(0, 0)$. For other vertices we set $e_1 = e_\rightarrow - e_\nwarrow = e^{-i\pi/6}$ and $e_2 = e_\nwarrow  - e_\swarrow = i$ and we say that both $w_0 + m e_1 + n e_2$ and $b_0 + m e_1 + n e_2$ have coordinates $(m, n)$.
 To an oriented edge of $\cH$, we associate the dual edge obtained by rotating it by $\pi/2$ in the positive direction, i.e we will write $(bw)^\dagger$ for the dual edge crossing $(bw)$ with the white vertex of its right and $(wb)^\dagger$ for the one crossing $(bw)$ with $w$ to its left. On $\cH^\dagger$, we say that the face $v_0$ immediately below $(w_0 b_0)$ has coordinates $(0, 0)$ and that $v_0 + m e_1 + n e_2$ has coordinate $(m,n)$ as in the primal lattice. We will write $w(m,n), b(m,n)$ and $v(m,n)$ for respectively the white, black and dual vertices of coordinates $(m,n)$.

\begin{figure}
\begin{center}
\includegraphics[width=.4\textwidth]{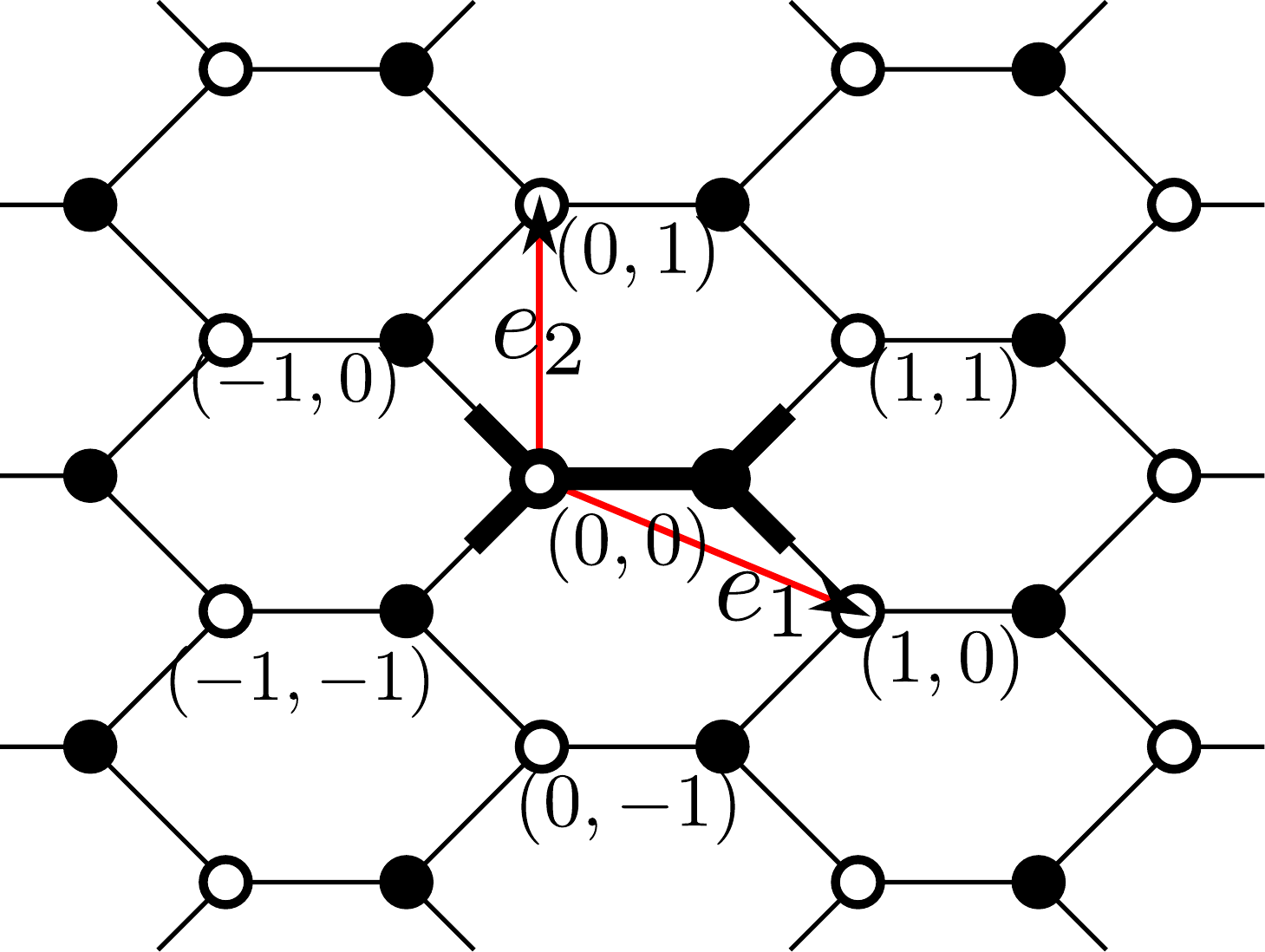}
\caption{The coordinate system used in this paper. We emphasize that the black and white vertices on top of each other share the same integer coordinate. The piece of the lattice drawn with heavier line is the fundamental domain.}\label{fig:coordinate}
\end{center}
\end{figure}

Recalling that we see $\cH$ as embedded in $\C$ with unit length regular hexagon, we will denote $\delta \cH$  by $\cH^\d$ for $\delta \ge 0$ and we will still use the same integer coordinates for points in $\cH^\d$, i.e on $\cH^\d$ we set $w(w,n) = w_0 + \delta m e_1 + \delta n e_2$ and similarly for black and dual vertices. Finally for any subset $U$ of $\C$, we will let $U^\d$ denote the subgraph of $\cH^\d$ induced by the intersection of $U$ and the vertex set of $\cH^\d$, which we will see as the naïve discretisation of $U$. We emphasize the difference in notation with $U(\delta)$ in \cref{thm:intro,prop:main} which will also denote a form of discretisation of $U$ but with a non trivial procedure. We will use both notation to emphasize the difference between the objects obtained with a naïve discretisation and the others throughout the paper.

We will say that horizontal edges have type $a$, edges in the direction $\pm e^{2i\pi/3}$ have type $b$ and edges in the direction $\pm e^{-2i\pi/3}$ have type $c$. 
In this paper, unless specified otherwise we will use the reference flow sending $1/3$ from every white vertex to the corresponding white vertex (see for example Section 3 in \cite{Tiliere2014} for details). For the ergodic Gibbs measures with probability $p_a, p_b, p_c$ depending on the type, this corresponds to the convention
\begin{gather}\label{eq:grad_to_proba_123}
\E( h(z+e_1) - h(z) ) = p_c - \frac{1}{3}, \quad \E( h(z+e_2) - h(z) ) = p_a - \frac{1}{3} \nonumber \\ \E( h(z-e_1 -e_2) - h(z) ) = p_b - \frac{1}{3}.
\end{gather}
if $z$ is the centre of a face. Going back to the usual coordinates on $\C$ gives
\begin{equation}\label{eq:grad_to_proba_xy}
\E( h( x + iy) - h(0)) =\frac{\sqrt{3}}{3}x(p_c - p_b) +  y(p_a - \frac{1}{3} )
\end{equation}
and in the other direction
\begin{equation}\label{eq:probo_to_grad}
p_a = \frac{1}{3} + \partial_y \E h, \quad p_b =- \frac{\sqrt{3}}{2} \partial_x \E h -\frac{1}{2} \partial_y \E h , \quad p_c =\frac{\sqrt{3}}{2} \partial_x \E h -\frac{1}{2} \partial_y \E h .
\end{equation}
In the following, we will identify the gradient of a height function with a triplet $p_a, p_b, p_c$ freely thanks to the above relation.

Finally we will sometime write $K(w,b) = 1_{w \text{ adjacent to } b}$ which we think of as the Kasteleyn matrix of the hexagonal lattice. This is mostly used in order to write statement that would make sense for dimers on a different lattice even though we are only interested in the hexagonal lattice here.

\subsection{General properties of T-graphs}\label{sec:general_Tgraph}

In this section we review the general definition of T-graphs and the link between uniform spanning tree and dimers. 

\begin{definition}
A finite (proper) T-graph $T$ is a finite collection of open segments $S_i$ and singles points $x_i$ in $\C$ such that
\begin{itemize}
\item all segments and points are disjoint,
\item $(\cup_i S_i) \bigcup (\cup_j x_j)$ is a closed connected set,
\item all $x_i$ are on the boundary of the unique unbounded connected component of $\C \setminus (\cup S_i \cup_j x_j )$.
\end{itemize}
The $x_i$ are called the boundary vertices of $T$ and we generally identify $T$ and $(\cup_i S_i) \bigcup (\cup_j x_j)$.
\end{definition}

\begin{figure}
	\begin{center}
		\includegraphics[width=.5\textwidth]{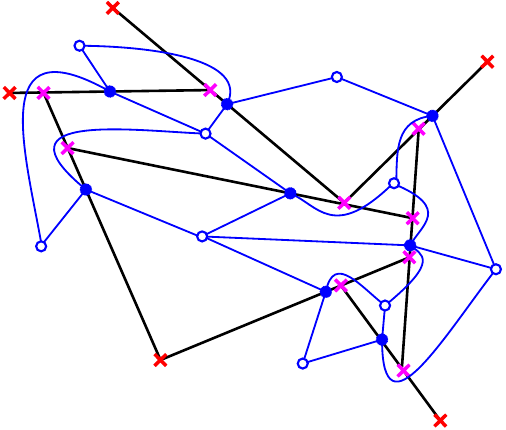}
		\caption{A T-graph and associated dimer graph. The T-graph is pictured with black segment, with its interior (resp. boundary) vertices marked by pink (resp. red)crosses. The dimer graph in blue is drawn with each black vertex on the corresponding segment, each white vertex inside the corresponding face and so that each face contains the corresponding vertex of T.}\label{fig:TgraphExemple}
	\end{center}
\end{figure}

The name T-graph comes from the fact that the endpoints of all segments have to be inside another segment (or a boundary vertex) like in the letters T or {\fontfamily{ptm}\selectfont K}. For simplicity and without loss of generality we will also assume that the $(x_i)_{1 \leq i \le n}$ are numbered in positive order when tracing the outer boundary of $T$. A set of open segments and points satisfying the first two assumptions but not the third is called a T-graph with degenerate faces and the interior $x_i$ are called degenerate faces.

\begin{definition}\label{def:dimers}
The dimer graph associated to a T-graph $T$ is defined as follows. Each segment $S_i$ in $T$ is a black vertex. Each finite connected component of $\C \setminus T$ is a white vertex, together with $n-1$ extra white vertices associated with the pairs $(x_i, x_{i+1})_{1 \leq i \leq n-1}$ which we will call boundary white vertices. We define weights and adjacency by setting $w( S, f )$ to be the length of $S \cap \bar f$ when $S$ is a segment and $f$ is an face (note that boundary white vertices are naturally associated to subsets of the outer face so the definition makes sense even for them).
\end{definition}

We also give a graph structure to a T-graph as follows. The set of vertices is the set of all endpoints of at least one segment (if two segments have the same endpoint like in a {\fontfamily{ptm}\selectfont K} we still see it as a single vertex). Two vertices are adjacent if they both belong to the same \emph{closed} segment and there is no other vertex between them.
 Clearly $T$ is then a planar embedding of this graph. We call all non-boundary vertices interior (note that an interior point can be located on the outer boundary of the infinite connected component of $\C \setminus T$ however, in fact the example of \cref{fig:TgraphExemple} is so small that this is the case for all interior vertices).
\begin{definition}
The random walk $X_t$ on $T$ is the pure jump Markov process defined by the following transition rates :
\begin{itemize}
\item the boundary vertices are sinks, transition out of them are impossible.
\item if $x$ is an interior vertex, then there is a unique open segment $S$ such that $x \in S$. Call $x^+$ and $x^-$ the two adjacent vertices of $x$ that are in $\bar S$. The only transition out of $x$ have rate $p( x \to x^\pm ) = \frac{1}{|x^\pm - x|\cdot|x^+ - x^-|}$.
\end{itemize}
\end{definition}
Note that this random walk is not reversible: Suppose that a segment $S$ has $n$ interior vertices $x_1, \ldots, x_n$ and call its two endpoints $x_0$ and $x_{n+1}$. Then a random walk started in $x_1$ can possibly move back and forth between $x_1, \ldots, x_n$ but it will eventually hit either $x_0$ or $x_{n+1}$ and from there it cannot jump back directly to any other $x_i$.
It is therefore natural to think of $T$ as a \emph{directed} graph. 

\begin{remark}\label{rq:trace_variance}
By construction the walk $X_t$ is a martingale. A simple computation shows that as long as $X$ has not reached a sink, $\frac{d}{dt} \Tr( \Var (X_t ) ) = 1$ because the expected time for the walk to leave $x$ matches the time for a Brownian motion started in $x$ and moving along the segment $S$ to reach $x^\pm$. 
\end{remark}

Once we have defined a random walk, we can define a wired uniform spanning tree measure, which we will in the following just call UST measure. 
\begin{definition}
We call wired spanning tree of $T$ (or just tree for short) any set $\cT$ of oriented edges of $T$ such that each interior vertex has a single outgoing edge in $\cT$ and $\cT$ does not contain any oriented cycle (not that non-oriented cycles are forbidden by the first condition). The wired spanning tree measure is the probability measure defined by:
\[
\P( \cT ) \propto \prod_{e \in \cT} w(e) = \prod_{(xy) \in \cT} \frac{1}{|x - y|}.
\]
It is well known that this measure can be sampled using Wilson's algorithm which we will not recall for brevity.
\end{definition}

We note that we changed the time parametrisation of the walk compared to the original definition in \cite{Kenyon2007b}. This does not matter when applying Wilson's algorithm and our convention is a bit nicer because of the exact variance computation given in \cref{rq:trace_variance}.
One of the main results of \cite{Kenyon2007b} is a map from UST on $T$ to dimers on the associated graph.
\begin{theorem}\label{the:KenyonSheffield}
Let $T$ be a finite T-graph and let $G$ be the associated dimer graph. The following construction maps any spanning tree of $T$ to a dimer configuration of $G$ and sends the UST measure to the dimer measure.

Let $x_1, \ldots, x_n$ be the boundary vertices of $T$. For each tree $\cT$, we consider its planar dual $\cT^\dagger$, which is a tree on the union of the faces of $T$ together with one face for each cyclical interval $(x_i, x_{i+1})$. Clearly we can see $\cT^\dagger$ as a tree on the set of white vertices of $G$ plus a single extra vertex $w_0$. We orient $\cT^\dagger$ towards $w_0$, and the map is defined from this orientation edge by edge: Suppose $f \to f'$ is an edge of $\cT^\dagger$, then in $T$ there exists a single segment $S$ such that $\bar f \cap \bar f' \subset S$ (where we see faces as sets in this expression). In the dimer configuration we match the white vertex associated to $f$ with the black vertex associated to $S$.
\end{theorem}

The actual construction of the mapping will not be used much in this paper, we will actually use only two non-trivial facts about this mapping. First we will have to read some partial information about where black vertices are matched in a simple way:

\begin{corollary}\label{cor:mapping_black}
Let $S$ be a segment of $T$ and let us enumerate the vertices of $\bar S$ from one endpoint to the other as $x_0, x_1, \ldots, x_{n+1}$. In any spanning tree of $T$, there exists a single $j$ such that $\cT$ contains all the oriented edges $x_i \to x_{i-1}$ for $i \le j$ and all the edges $x_i \to x_{i+1}$ for $i \geq j+1$, or in other word only a single interval $(x_j, x_{j+1})$ is not covered by $\cT$. Let $w$, $w'$ be the two white vertices associated to the two faces adjacent to $(x_j x_{j+1})$. In the mapping to dimers, the black vertex associated to $S$ is matched with either $w$ or $w'$.
\end{corollary}

Second, we will need to read the height function directly from the tree. This relies on choosing a proper convention of the definition of the height function and will therefore be treated in \cref{sec:flow} for our specific case.

\subsection{The family $T_{\Delta, \lambda}$}

In this section, we define a particular family of whole plane T-graphs which are associated to whole plane Gibbs measures on dimers. 
In the context of this paper, they will be used to describe the local behaviour of our ``main'' T-graph $\Gamma(\delta)$ (which is not very surprising considering that dimer measures are known \cite{Aggarwal2019} to behave locally like whole plane ergodic Gibbs measure). This family is indexed by two parameters, a triplet $(A,B,C)$ of points in the complex plane which we see as a triangle and call $\Delta$, and a complex number in the unit circle $\lambda$. We say that $\Delta$ is a flat triangle if $A,B$ and $C$ are aligned (including the case where two points are equal) and we will in general consider only non-flat triangles.

\begin{definition}\label{def:Tgraph_plan}
In the following, we will denote by $\Delta$ a triplet $(A,B,C)$ of distinct complex numbers which we see as a triangle with vertices $A, B,C$ in positive order. Given $\Delta$, we define functions $F_\Delta$ and $G_{\Delta}$ on respectively white and black vertices of $\cH$ by
\begin{gather*}
F_\Delta ( w(m,n) ) = (\frac{A - C}{C - B})^m (\frac{(B - A)}{(A - C)})^n, \\
G_\Delta (b(m,n)) = (C-B) (\frac{A - C}{C - B})^{-m} (\frac{(B - A)}{(A - C)})^{-n}.
\end{gather*}
Given $\lambda$ in the unit circle, we can now define two flows $\Omega_{\Delta, \lambda}$ and $\Omega_{\Delta, \lambda}^\dagger$ on $\cH$ and $\cH^\dagger$ by
\begin{gather*}
\Omega_{\Delta, \lambda} (bw) = \Omega^\dagger_{\Delta, \lambda}(bw^\dagger) = 2 \Re( \lambda F(w) ) \bar \lambda G(b), \\
\Omega_{\Delta, \lambda} (wb) = \Omega^\dagger_{\Delta, \lambda}(wb^\dagger) = - \Omega_{\Delta, \lambda} (bw).
\end{gather*}
Finally, we define the map $T_{\Delta, \lambda}$ as the primitive of $\Omega^\dagger_{\Delta, \lambda}$, i.e $T_{\Delta, \lambda}$ is the unique map from $\cH^\dagger$ to $\C$ such that 
\[
\forall v \sim v' \in \cH^\dagger, T_{\Delta, \lambda}(v') - T_{\Delta, \lambda}(v) = \Omega_{\Delta, \lambda}^\dagger( v v' ),
\]
and $T_{\Delta, \lambda}(v_0) = 0$ where recall that $v_0$ is the dual vertex of coordinates $(0,0)$. Note that the existence is not immediate but follows from a simple computation. In general, we think of $T$ as extended linearly along the edges of $\cH^\dagger$ and we will identify the map $T_{\Delta, \lambda}$ with its image so we will think of $T_{\Delta, \lambda}$ as a closed connected set.
\end{definition}

We will often omit the subscript from $T$ when there is no risk of confusion. Also we will allow ourself to write $T(w)$ and $T(b)$ for respectively the image of a white and black face of $\cH^\dagger$, in both cases thinking about $T$ as extended linearly along edges of $\cH^\dagger$.

We now state a few facts about the geometry of $T$, together with short sketch of proofs (see \cite{Laslier2013b} for the complete proofs). We think that these sketches will be useful both to understanding of T-graphs in general and because close variants of these proofs will appear during the construction of $\Gamma(\delta)$ in \cref{sec:construction}.

\begin{proposition}\label{prop:geometryTgraph}
For all $\Delta$ and $\lambda$, the graph $T_{\Delta,\lambda}$ satisfies the following:
\begin{itemize}
\item For all white vertex $w$, $T( w ) = \Re(\lambda \frac{F(w)}{|F(w)|} ) \bar \lambda \Delta$, up to translation;
\item For all black vertex $b$, $T( b ) = \Re(\lambda \bar \Delta ) \bar \lambda \frac{G(b)}{|G(b)|}$, up to translation;
\item For all $m,n$, we have $T(v(m,n) ) = m (B-A) + n(C - B) + O(1)$,
\item For any $w, w'$, the interior of $T(w)$ and $T(w')$ are disjoint, similarly the images $T(b)$ do not intersect except at their endpoint;
\item For any dual vertex $v$, only the following two geometries are possible in a neighbourhood of $T(v)$. The first possibility is that $T(v)$ is in the interior of exactly one $T(b)$ and an endpoint of exactly two others. The second possibility is that $T(v)$ is an endpoint of exactly $6$ segments and an interior point of none, in that case there is a white vertex $w$ such that $T(w) = T(v)$. For generic $\lambda$, the second case never happens.
\end{itemize}
\end{proposition}
\begin{proof}[Sketch of proof]
The first two points are just by inspection of the increments. 

For the third point, writing the increments along a vertical line, we have
\[
T(v(0, n) ) = \sum_{j=0}^{n-1} F(w(0, j))G(b(0, j)) + \bar \lambda^2 \bar F(w(0, j)) G( b(0, j) ).
\]
The first term in the sum is constant and gives $n(C-B)$ while the second term has constant modulus but its argument is rotating with $j$ and therefore it gives a bounded contribution.

For the fourth point, we note that by the first point all white faces of $\cH^\dagger$ are mapped to $C$ preserving their orientation. Suppose that there is a point $x$ in $T(w) \cap T(w')$ but not in any $T(b)$, then for any cycle $\cC$ surrounding both $w$ and $w'$, $T(\cC)$ must be a cycle doing al least two turns around $x$. Taking a large enough cycle gives a contradiction with point $3$.

Finally for the last point, first we note that by the first point, if $\Re( \lambda F(w) ) \neq 0$ for all $w$ then all dual edge are mapped to segments of non-zero length. In that case, it is easy to check from a local computation that the alternating product of the directions of all edges around a vertex $v$ must be $1$. In particular, each $v$ must be in the interior of either $1$ or $3$ of the $T(b)$, but $3$ is impossible since the $T(b)$ do not intersect. The case where a white face is mapped to a point is done by a local computation.
\end{proof}

This proposition clearly says that the $T_{\Delta, \lambda}$ are (infinite volume generalisation of) T-graphs with possible degenerate faces. When there is no degenerate face, the map from tree to dimers has a straightforward generalisation and unsurprisingly it can be proved that the uniform spanning tree on $T_{\Delta, \lambda}$ is mapped to an infinite volume ergodic Gibbs measure on the hexagonal lattice. The fact that there \emph{can} be degenerate faces will however be one of the main technical issue the we will have to work deal with, see in particular \cref{lem:choice_lambda} and \cref{sec:cltTgraph}.

\subsection{Property of the random walk}\label{sec:RW}

In this section, we recall a few results on the random walk on T-graphs that will be needed later.

The first of these is a version of Berstein inequality for our continuous time martingale.
\begin{proposition}[Proposition 6.1 in \cite{Chelkak2020}]\label{prop:concentration}Let $X_t$ be the random walk on a T-graph and assume that all segments of the graph have length at most $L$, then for all $t$ and $c$ we have
\[
\P\big(\sup\nolimits_{s\in [0,t]}|X_s-X_0|\ge 2 c \sqrt{t}\,\big)\ \le\ 4\exp\big(-\tfrac{1}{2}c^2\cdot(1+\tfrac{2}{3} L c t^{-1/2}\,)^{-1}\big)\,.
\]
{In particular, the left-hand side is exponentially small in~$c$ uniformly over all~$t\ge L^2$.}
\end{proposition}

Together with the concentration, we will also use the ``reverse'' fact meaning that the walk is unlikely to remain stuck in a small region for a long time.
\begin{lemma}\label{lem:exit}
Let $X_t$ be the random walk on a T-graph and assume that all segments of the graph have length at most $L$. For $r \ge L$, we let $\tau_{r}$ be the first time $X_t$ exits $B( 0 , R)$  or reaches a boundary vertex (assuming $X_0 \in B( 0, r)$), then for all $n \geq 1$,
\[
\P( \tau_R \geq 18 n r^2 ) \leq 2^{-n}.
\] 
\end{lemma}
\begin{proof}
This is a simple consequence of the fact that the trace of the variance grows linearly with time, meaning that $|| X_t - X_0 ||^2 - t$ is a martingale until the walk reaches a boundary point. Clearly $\E( || X_{\tau_r} - X_0 ||^2) \leq 9r^2$ and therefore $\P( \tau_r \geq 18 R^2 ) \leq 1/2$. We can then obtain the result simply by iterating.
\end{proof}

For the graphs $T_{\Delta, \lambda}$, we also know that the walk cannot become one dimensional.
\begin{lemma}[Proposition 2.22 in \cite{Laslier2013b}]\label{lem:variance}
For all non flat $\Delta$ there exists $\epsilon >0$ such that for all $\lambda$ for all $x \in T_{ \Delta, \lambda}$, the random walk started from $X$ satisfies
\[
\forall d \in \bbS_1, \, \epsilon \leq \Var( X_1 . d ) \leq 1.
\]
Furthermore $\epsilon$ can be chosen to depend continuously on $\Delta$.
\end{lemma}

 Note that the result still holds for any graph where the law of the random walk is close enough to the law on one of the $T_{\Delta,\lambda}$.

From this we can directly obtain a so called uniform crossing estimate.
Given~$z\in\C$ and~$r>0$, define
\[
\cR(z,r)\ :=\ {z +[-3r,3r]\times[-r,r]},\quad B_1(\cR):=B(z\!-\!2r,\tfrac{1}{2}r),\quad B_2(\cR):=B(z\!+\!2r,\tfrac{1}{2}r)
\]
i.e. $R$ is an horizontal rectangle with $3 \times 1$ aspect ration, $B_1$ is a ball in the left third of $R$ and $B_2$ a ball in the right third. 
\begin{lemma}[Theorem 3.8 in \cite{BLRannex}]\label{lem:uniform_crossing}
Consider a T-graph satisfying \cref{lem:variance} and with segments of size at most $L$. There exists constants~$\rho'_0,\varsigma'_0>0$ such that the following holds for all rectangles~$\cR(z,r)$ with $r\ge\rho'_0 L$ drawn over the T-graph:
\[
\P^v ( \text{$X_t$ hits $B_2(\cR)$ before exiting $\cR$} )\ \ge\ \varsigma'_0\quad \text{for all $v\in B_1(\cR)$}.
\]
The same result also holds with the additional constraint that the $X_t$ hits $B_2(\cR )$ before time $C r^2$ for some constant $C$.
\end{lemma}

Finally another consequence of the bound on the variance and the martingale property is that the random walk is unlikely to spend all its time on a small subset of the vertices. 
\begin{lemma}\label{lem:nb_visite}
Consider a T-graph satisfying \cref{lem:variance} and with segments of size at most $L$. There exists a constant $C$ depending only on $L$ and the $\epsilon$ from \cref{lem:variance} such that the following holds for any $R \geq L$. For all $x \in B(x, R)$ and for all functions $f$ on $B(x, R)$ we have
\[
\sup_{v \in B(x, R)} \E_{v} [ \sum_{0 \leq k < \tau} f( X_k )] \leq C R^2 \Big(\frac{1}{R^2} \sum_{v \in B(x, R)} f^2(v') \Big)^{\frac 1 2},
\]
where in the sum and the supremum we consider only vertices of the T-graph. and $\tau_{B(x, R)}$ denotes the exit time from the ball.
\end{lemma}
This is basically Lemma 3.21 in \cite{Laslier2013b} which itself follows Sznitman but we provide a full proof to emphasize that the geometry of the graph does not matter beyond the assumptions of the lemma.
\begin{proof}
Let us denote by $B$ the set of vertices of the T-graph inside $B(v, R)$ to simplify notations. 
Let $\partial B$ denote the outer boundary of $B$ so that $\tau$ is the hitting time of $\partial B$. Finally we write $u(v) = \E_{v} [ \sum_{0 \leq k < \tau} f( X_k )]$ for $v \in B$, with by convention $u(v) = 0$ for $v \in \partial B$. Note that $u$ satisfies $\E_v (u(X_1) - u(x) ) = - f(x)$.
Let $s (v) = \{ {\bf c}  \in \bbR^2 | \forall v' \in B \cup \partial B, u(v') \leq u(v) + {\bf c} \cdot (v'-v) \}$ and let $S = \cup_{v \in B} s(v)$. 

 We start by giving a lower bound on the volume of $S$.
 Fix ${\bf c} \in \bbR^2$ such that $\abs{{\bf c}} < \max(u)/(2R+L)$ and let $v_0 \in B$ be a point where $\max(u)$ is attained. Clearly by construction and since $u=0$ on $\partial B$, for all $v \in \partial B$
\[
  u(v_0) + {\bf c}\cdot(v-v_0) - u(v) > 0.
\]
On the other hand, the function $v \to u(v_0) + {\bf c}\cdot(v-v_0) - u(v)$ has value $0$ at $v_0$ and therefore must reach its minimum at some $v_{\min} \in B$. it is then easy to see that ${\bf c} \in s(v_{\min})$ and so ${\bf c} \in S$. Overall we have proved that $S$ has at least volume $\max(u)^2/(2R+L)^2$.

Now we will upper bound the volume of $S$ by giving an upper bound on the volume of each $s(v)$. Fix $v \in B$, ${\bf c} \in s(x)$ and a set $V'$ such that $\bbP_v(X_1 \in V')=p > 0$. 
Since ${\bf c} \in s(v)$, the random variable $u(v) - u(X_1) + {\bf c}\cdot(X_1 -v)$ is positive and thus
\[
  \E_v[ u(v) - u(X_1) + {\bf c}\cdot(X_1 -v) ] \geq p \inf_{v' \in V'}\left(u(v) - u(v') + {\bf c}\cdot(v' -v)\right).
\]
Since the walk on a T-graph is a martingale, we have
$
  \E_v[ {\bf c}\cdot(X_1-v) ] =0 ,
$
and since $\E_v[u(X_1) - u(v)]=-f(v)$, we can rewrite for some $v' \in V'$
\[
  {\bf c}\cdot(v' -v) \leq  u(v')-u(v) + f(v)/p .
\]
On the other hand, we also have by applying directly the definition of $s(v)$ to $v'$ :
\[
  {\bf c}\cdot(v' -v) \geq u(v')-u(v).
\]
This shows that ${\bf c}$ is in a strip oriented in the orthogonal direction to $v'-v$ and of width $f(v)/p|v'-v|$ .
Finally, using \cref{lem:variance} and the fact that segments have at most length $R$, we see that we can always find two strips with different directions and a lower bound on the width. This proves that the volume of $s(v)$ is at most $C f^2(v)$ for all $v$.
Comparing with the previous lower bound on the volume of $S$, we obtain
\[
\frac{\sup_v u^2(v)}{(2R+L)^2} \leq C \sum_{v} f^2(v),
\]
which is the desired expression for $R \geq L$.
\end{proof}

We conclude the section with the central limit result from \cite{Laslier2013b}.

\begin{theorem}[Theorem 2.18 in \cite{Laslier2013b}]\label{thm:CLT_ref}
Let $\Delta$ be a triangle in the plane. For almost all $\lambda$, let $X_t$ denote the random walk on $T(\Delta, \lambda)$ started in the origin. As $\delta$ goes to $0$, the process $\delta X_{t/\delta^2}$ converges to a Brownian motion whose covariance matrix is proportional to the identity and does not depend on $\lambda$ (it does depend of $\Delta$).
\end{theorem}

\section{Construction of the graph}\label{sec:construction}

The goal of this section is to construct T-graphs $\Gamma(\delta)$ such that the associated dimer graph are subsets of the hexagonal lattice (we will call them $U(\delta)$) with the prescribed limit height function $h^\cC$.

As mentioned in \cref{sec:sketch}, the construction follows the idea from \cite{Kenyon2007} with more precision. The first step is to construct two functions $F$ and $G$ which will play the role of $F_{\Delta}$ and $G_\Delta$ from \cref{def:Tgraph_plan}. In particular they will be designed so that $\Re(F) G$ almost defines a closed form from which one can take a sort of primitive $\psi$ that looks like a T-graph, albeit with some unavoidable error. However the very geometric nature of the definition of a T-graph makes it possible to fairly easily to correct $\psi$ into a proper T-graph. In the spirit of the switch from \cref{thm:intro} tp \cref{prop:main}, we will avoid having to treat boundary carefully by doing the construction is a slightly larger domain and cutting a portion of it and we keep subscript $ex$ to differentiate between objects associated to $U$ and $U_{ex}$.

\begin{notation}
Starting from this section, we will extensively use $O( . )$ notations. Unless explicitly specified otherwise, the implied constants depends only on firstly the distance between $\nabla h^\C (U_{ex} )$ and the set of allowed slopes, secondly uniform bounds on the derivatives of $h^\cC$ and the function $\phi$ defined in \cref{prop:def:phi} up to order $31$, and third the area of $U_{ex}$. We also use the notation $f = \Theta (g )$ to say that $f/g$ is bounded and bounded away from $0$, with the same dependence as above for the constant. Finally ``away from the boundary'' will always mean for all points at distance at least $\Theta(\delta )$ of the boundary.
\end{notation}

\begin{remark}
	We will have two types of control in this section which might seem redundant but this is necessary because \cref{sec:CLT} and \cref{sec:comparison} require very different estimates. For the CLT in \cref{sec:CLT}, we will care about approximating large microscopic region but since we do not need any control of the speed in the CLT the precision of these approximation will not be important. On the opposite in \cref{sec:comparison} we will only look at the scale of a single segment or face but error terms will have to be very small.
\end{remark}

\subsection{First approximation from the continuum}\label{sec:definition_flow}

\begin{definition}\label{def:Phi}
We let $\Phi : U_{ex} \to \bbH$ be the top of the triangle associated to $\nabla h_\infty$, in the sense that for any $z$, $\Phi(z)$ is the third point in $\bbH$ of the triangle whose two other vertices are $0$ and $1$ and whose angles are $\pi p_a, \pi b, \pi p_c$ in positive order starting from the angle at $\Phi$. By the sine law we have
\[
\Phi = \frac{\sin(\pi p_c )}{\sin( \pi p_a)} e^{i \pi p_b}
\]
and in particular $\Phi$ is in $\cC^{32}.$
\end{definition}

\begin{lemma}[Theorem 1 in \cite{Kenyon2007a}]\label{lem:complex_burgers}
The function $\Phi$ satisfies the equation
\[
\partial_1 \Phi + \Phi \partial_2 \Phi = 0,
\]
where $\partial_1$ and $\partial_2$ denote partial derivatives in the direction of $e_1$ and $e_2$. In terms of the usual partial derivatives it becomes
\[
\frac{\sqrt{3}}{2} \partial_x \Phi + (\Phi - \frac{1}{2}) \partial_y \Phi = 0.
\]
\end{lemma}
\begin{proof}We provide a proof of this result despite the fact that it is not new because it is a fairly straightforward computation and it makes the paper more self-contained. Also the formulation is somewhat different.
This is a reformulation of the equation satisfied by an asymptotic height function:
\[
\sum_{i, j} \sigma_{ij}( \nabla h_\infty ) \partial_{ij} h_\infty = 0,
\]
with $\sigma = \Lambda( \pi p_a) + \Lambda( \pi p_b) + \Lambda( \pi p_c)$ where $\Lambda(x)$ is the primitive of $\log( 2 \sin(x) )$. Writing all derivatives in terms of $p_a$, $p_b$, $p_c$, it reads
\begin{multline*}
\frac{\sqrt{3}}{4}\Big( \cot( \pi p_b) + \cot (\pi p_c) \Big)(\partial_x p_c - \partial_x p_b) + \Big(\cot( \pi p_a) + \frac{1}{4} \cot( \pi p_b) + \frac{1}{4} \cot (\pi p_c)  \Big)\partial_y p_a \\
+ \frac{\sqrt{3}}{4} \Big( \cot (\pi p_b) - \cot (\pi p_c) \Big)\Big( \partial_x p_a + \partial_y \frac{\sqrt{3}}{3} (p_c - p_b)  \Big) = 0.
\end{multline*}
Using that $p_a + p_b+p_c = 1$, this can be simplified as
\[
\frac{\sqrt{3}}{2}( \cot \pi p_c \partial_x p_c - \cot \pi p_b \partial_x p_b) - \frac{1}{2} (\cot \pi p_b \partial_y p_b + \cot \pi p_c \partial_y p_c ) + \cot \pi p_a \partial_y p_a = 0,
\]
which in turn we can restate in terms of the directional derivatives as
\[
\cot \pi p_c \partial_1 p_c + \cot \pi p_a \partial_2 p_a + \cot \pi p_b (- \partial_1 - \partial_2) p_b = 0.
\]

On the other hand,
\[
\partial_1 \Phi = \partial_1 p_c \cot(\pi p_c) \Phi - \partial_1 p_a \cot (\pi p_a) \Phi + i \partial_1 p_b \Phi,
\]
and similarly for $\partial_2 \Phi$.
We get
\begin{multline*}
\Re(\frac{1}{\Phi}\partial_1 \Phi + \partial_2 \Phi ) = \partial_1 p_c \cot(\pi p_c)  - \partial_1 p_a \cot (\pi p_a) \\
+\frac{\sin(\pi p_c)}{\sin (\pi p_a)} \Big( \partial_2 p_c \cot(\pi p_c)  \cos(\pi p_b)  - \partial_2 p_a \cot (\pi p_a) \cos(\pi p_b) - \partial_2 p_b \sin(\pi p_b)  \Big).
\end{multline*}
It is not hard to check that this means
\begin{multline*}
\Re(\frac{1}{\Phi}\partial_1 \Phi + \partial_2 \Phi ) = \frac{\cos( \pi p_c) \sin \pi p_b}{\sin( \pi p_a)} \Big(\cot \pi p_c \partial_1 p_c + \cot \pi p_a \partial_2 p_a + \cot \pi p_b (- \partial_1 - \partial_2) p_b  \Big) \\
- \frac{\sin \pi p_b \sin \pi p_c}{\sin \pi p_a} \Big( \partial_2 \pi p_b + \partial_1 \pi p_a + \partial_2 \pi p_a \Big).
\end{multline*}
The first parenthesis is obviously the equation characterising an asymptotic height. The second parenthesis is $0$ because of the commutativity of directional derivatives of $h_\infty$.

For the imaginary part, we get
\begin{align*}
\Im(\frac{1}{\Phi} &\partial_1 \Phi + \partial_2 \Phi )   \\
&= \partial_1 p_b + \frac{\sin \pi p_c}{\sin \pi p_a}\Big(\partial_2 p_c \cot(\pi p_c) \sin(\pi p_b) - \partial_2 p_a \cot (\pi p_a) \sin( \pi p_b) + \partial_2 p_b \cos( \pi p_b) \Big) \\
& = \frac{\sin( \pi p_b) \sin( \pi p_b) }{\sin(\pi p_a)} \Big( (\partial_1+ \partial_2) p_b \cot( \pi p_b) - \partial_1 p_a \cot( \pi p_a) + (\partial_1 p_b + \partial_2 p_c ) \cot( \pi p_c)\Big),
\end{align*}
which is indeed equal to $0$ as above.
\end{proof}

\begin{proposition}[Section 4.1 in \cite{Kenyon2007} ]\label{prop:def:phi}
There exists two functions $H : U_{ex} \to \bbC$ and $\phi : U_{ex} \to \bbD$ such that :
\begin{itemize}
\item $d H = [2i \pi- \log( \Phi) - \log( 1 -\Phi) ]\frac{\sqrt{3}}{3}dx + (\log( 1 -\Phi)  - \log (\Phi))d y$; 
\item $\phi$ is a $\cC^{32}$ diffeomorphism between $U_{ex}$ and the unit disk $\bbD$;
\item $\frac{\frac{d\phi}{d \bar z}}{\frac{d\phi}{dz}} =
 \frac{\frac{d\Phi}{d \bar z}}{\frac{d\Phi}{dz}} = \frac{\Phi - e^{i\pi/3}}{\Phi - e^{-i\pi/3}}$.
\end{itemize}
where the arguments are always taken with $\arg(\Phi) \in (0, \pi)$ and $\arg( 1 - \Phi) \in (-\pi, 0)$.
\end{proposition}
\begin{proof}
For the first item, we need to check that the form is closed, i.e whether
\[
\partial_y ( -\log( \Phi -\Phi^2) \frac{\sqrt{3}}{3}  ) - \partial_x \log( \frac{1}{\Phi} -1) = 0 .
\]
The left hand side gives
\[
-\frac{\partial_y \Phi( 2 \Phi -1)}{\Phi - \Phi^2}\frac{\sqrt{3}}{3} - \frac{- \partial_x \Phi/ \Phi^2}{\frac{1}{\Phi} - 1} \\
 = \frac{1}{\Phi - \Phi^2} \Big(\partial_y \Phi( 1 - 2 \Phi)\frac{\sqrt{3}}{3}  + \partial_x \Phi\Big)
\]
and this is indeed $0$ by the Burgers equation.

The two other items follow from the Ahlfors-Bers theorem (which is also called measurable Riemann mapping theorem) : The expression for  the Beltrami coefficient $\frac{\frac{d\Phi}{d \bar z}}{\frac{d\Phi}{dz}}$ is a direct consequence of \cref{lem:complex_burgers}. Then the fact that $ \Im( \Phi ) > 0$ guaranties that this coefficient is bounded by $1$ which is the assumption of the theorem. The fact that $\phi$ is as regular as $\Phi$ is a consequence of the theorem, see \cref{app:continuity} for details.
\end{proof}

From now on, we fix two functions $H$ and $\phi$ as given by the proposition.

\begin{remark}
It is easy to check that $\phi$ also satisfies the Burgers equation in the sense that
\[
\partial_1 \phi + \Phi \partial_2 \phi = 0.
\]
We can also assume without loss of generality that $\phi, H$ and all their derivatives extend continuously to the boundary of $U_{ex}$, and in particular that they are all bounded. Indeed if this was not the case we could just restrict ourself to a slightly smaller domain $\phi^{-1}( (1-\epsilon)\bbD  )$.
\end{remark}

\begin{lemma}\label{lem:square_root}
The function $\phi$ satisfies $0 \notin \frac{d\phi}{dy}( U_{ex} )$, in particular we can define $\sqrt{ \frac{d\phi}{dy} }$ continuously on $U_{ex}$.
\end{lemma}
\begin{proof}
Note that if $\frac{d\phi}{d y} = 0$ then we have $\frac{d \phi}{d z} = \frac{d \phi}{d\bar z}$ so in particular this can only happen if $|\frac{\Phi - e^{i\pi/3}}{\Phi - e^{-i\pi/3}}| = 1$. However this is impossible because $\Phi$ has strictly positive imaginary part.
\end{proof}

Now we can define the functions $F$ and $G$ from which we create our first approximation of a T-graph.
\begin{lemma}[Lemma 4.1 in \cite{Kenyon2007}]\label{lem:def_FG}
One can find two bounded functions $M^{+,\delta}$ and $M^{-, \delta}$ on white vertices of $U_{ex}^\d$ such that setting
\[
F(w) = e^{-H(w)/\delta} \sqrt{d_y \phi(w)} (1 + \delta M^{+, \delta}(w) )
\]
\[
G(b) = e^{H(w)/\delta} \sqrt{d_y \phi(w)} (1 - \delta M^{-, \delta}(w) )
\]
where $w$ in the definition of $G$ is the vertex just left of $b$.
We have
\[
G(b)\sum_{w \sim b} F(w) K(w, b) = O(\delta^{30})
\]
\[
F(w)\sum_{b \sim w} K(w, b) G(b) = O(\delta^{30}).
\]
\end{lemma}
\begin{proof}
The proof is essentially a computation by induction and we will only write it for $M^{+, \delta}$ since the other case is analogous. Our goal will be to write $\delta M^{+, \delta} = \sum_{n=1}^{30} M_{n} \delta^n$ where the functions $M_n$ will be obtained as solution of equations involving only continuum quantities.

 We will use the following derivatives of $H$ :
\begin{align*}
\partial_1 H 
& = i \pi - \log( 1 - \Phi ) \\
(\partial_1 + \partial_2) H 
& = i\pi - \log( \Phi ) \\
(\partial_1)^2 H &= \frac{\partial_1 \Phi}{1-\Phi} =\frac{- \Phi \partial_2 \Phi}{1 - \Phi}\\
(\partial_1 + \partial_2)^2 H 
& = \frac{ \Phi -1}{\Phi} \partial_2 \Phi
\end{align*}

To simplify notation and without loss of generality, we assume that the white vertex to the left of $b$ is $0$ so that the three neighbours of $b$ are $0, e_1$ and $e_1+e_2$. In the functions $\phi, M^{+, \delta}, \Phi$ and $H$, we drop the argument whenever it is $0$ and we drop the superscript in $M^{+ , \delta}$. We have
\begin{multline*}
\frac{G(b)\sum_{w \sim b} F(w)}{(1+ \delta M^{-, \delta})  } = (1 + \delta M) \partial_y \phi \\
+ (\Phi - 1)  e^{\frac{-H(e_1)+ H(0)}{\delta} + \partial_1 H} \sqrt{\partial_y \phi(e_1)\partial_y \phi }(1+ \delta M(e_1) ) \\
+ (-\Phi)\partial_y \phi e^{\frac{-H(e_1+e_2)+ H(0)}{\delta} + (\partial_1+\partial_2) H} \sqrt{\partial_y \phi(e_1+e_2)\partial_y \phi}(1+ \delta M(e_1+e_2) )
\end{multline*}
where we used the explicit expression of the derivatives of $H$ from above (we keep $\partial_y$ and $\partial_2$ separate to make it easier to keep track of the different terms).

We first argue that for any smooth function $M$, the contributions of order $0$ and $1$ in $\delta$ vanish. Indeed the order $0$ gives simply
\[
1 + (\Phi - 1) - \Phi = 0.
\]
The first order is
\begin{align*}
M( 1 & + (\Phi  - 1) - \Phi )\partial_y \phi + (\Phi - 1) \frac{-1}{2} \partial_1^2 H \partial_y \phi - \Phi \frac{-1}{2} (\partial_1+ \partial_2)^2 H \partial_y \phi \\
& \hspace{7cm}+ (\Phi - 1) \frac{1}{2} \partial_1 \partial_y \phi - \Phi \frac{1}{2}(\partial_1+\partial_2)\partial_y \phi \\
&= 0 - \frac{1}{2}  \Big( (\Phi - 1)\frac{- \Phi\partial_2 \Phi}{1 - \Phi} -\Phi\frac{(\Phi - 1) \partial_2\Phi}{\Phi} \Big)\partial_y \phi + (\Phi - 1) \frac{1}{2}  \partial_y(- \Phi \partial_2\phi) - \Phi \frac{1}{2} \partial_y ((-\Phi + 1)\partial_2 \phi) \\
&= - \frac{1}{2} \partial_2 \Phi \partial_y \phi - \frac{1}{2}(\Phi-1)\partial_y \Phi \partial_2 \phi  + \frac{1}{2} \Phi \partial_y \Phi \partial_2 \phi =0 \\
\end{align*}
where in the last line, we used that $\partial_y = \partial_2$ and we where just keeping the two notations to keep track of the different terms.

We now write $\delta M = \sum_{n=1}^{30} M_n \delta^n$ and start giving conditions that must be satisfied by the $M_n$ inductively. We note that the second order in $\delta$ in $G(b)\sum_{w \sim b} F(w) K(w, b)$ is given by
\begin{multline}
M_2 \big(1 +(\Phi - 1) - \Phi\big) \partial_y \phi + \big((\Phi - 1) \partial_1 M_1  - \Phi(\partial_1 + \partial_2) M_1 \big)\partial_y \phi \\
+ M_1 \Big( (\Phi - 1) \frac{-1}{2} \partial_1^2 H - \Phi \frac{-1}{2} (\partial_1+ \partial_2)^2 H  + (\Phi - 1) \frac{1}{2} \partial_1 \partial_y \phi - \Phi \frac{1}{2} (\partial_1+\partial_2)\partial_y \phi \Big)\\
 + J_2
\end{multline}
where $J_2$  is a polynomial expression involving only $\Phi$ and derivatives of $\phi$ and $H$ up to order $3$. Therefore the second order will vanish if and only if $M_1$ satisfies the equation
\[
\partial_1 M_1 + \Phi \partial_2 M_1 = \frac{J_2}{\partial_y \phi}.
\]
Similarly, in the third order, the only non-zero term involving either $M_2$ or $M_3$ is $(\partial_1 M_2 + \Phi \partial_2 M_2) \partial_y \phi$. Therefore the third order vanishes if and only if $M_2$ satisfies an equation of the form
\[
\partial_1 M_2 + \Phi \partial_2 M_2 = \frac{J_3}{\partial_y \phi}
\]
where $J_3$ depends on derivatives of $\Phi, \phi$ up to order $4$ as well as on $M_1$ and its first derivatives.

We can clearly go up to any order, so we just need to argue that all our equations $\partial_1 M_n + \Phi \partial_2 M_n = J_{n+1}/\partial_y \phi$ have a regular enough solution.

The first step is a change of variable. Fix $n$ and let $\cM_n = M_n \circ \phi^{-1}$, the equation for $M_n$ becomes
\[
\partial_x \cM_n (\partial_1 +\Phi \partial_2)\Re \phi + \partial_y \cM_n (\partial_1 +\Phi \partial_2)\Im \phi = \frac{J_{n+1}}{\partial_y \phi}
\]
and since $\phi$ satisfies $(\partial_1 + \Phi \partial_2) \phi = 0$, we get
\[
\partial_x \cM_n (\partial_1 +\Phi \partial_2)\Re \phi + i \partial_y \cM_n (\partial_1 +\Phi \partial_2)\Re \phi = \frac{J_{n+1}}{\partial_y \phi}
\]
which is an equation of the form $\partial_{\bar z} \cM = \mathcal{J}$ for some smooth $\mathcal{J}$, with $\cM$ and $\mathcal{J}$ defined in the unit disc. This is the well known inhomogenous Cauchy-Riemann equation (or d-bar equation) and by Theorem 4.7.2 in \cite{Astala2008} it has a solution which is continuous in the disc. We refer to \cref{lem:continuityMn} for details on the space of solution and dependence with respect to $h^\cC$.
\end{proof}

\begin{lemma}\label{lem:def_psi}
For any $\lambda$ in the unit circle we let $\Omega = \Omega_\lambda$ be the antisymmetric flow on edges defined by $\Omega( bw ) = -\Omega(wb) = 2 \delta \bar \lambda \Re( \lambda F(w)) G(b) K(w,b) $ and we let $\Omega^\dagger$ be the corresponding dual flow. There exists a function $\psi$ defined on the dual of $U_{ex}^\d$ such that away from the boundary $d \psi = \Omega^\dagger + O(\delta^{29})$. (The convention for the dual graph on the boundary is unimportant as we will only use $\psi$ away from the boundary).
\end{lemma}
\begin{proof}
We define $\psi$ by first taking an arbitrary point $v_0$ and fixing the value there arbitrarily. For any other points $v$, we take arbitrarily a simple path from $v_0$ to $v$ and just integrate $\Omega^\dagger$ along this path.

To control the derivative of $\psi$, we look at two of our arbitrary paths going to nearby points. If we deform a path on a single face, we change the value of the $\psi$ by $O(\delta^{30}\times \delta)$. The two paths can be deformed into each other by moving through at most $O(\delta^{-2})$ points so we are done.
\end{proof}

\subsection{Geometry of $\psi$}

Now that we have defined the map psi, we need to understand its geometry to be able to build a T-graph from it. More precisely, we will show that $\psi$ looks almost like a T-graph with quantitative bounds on the errors. This will allow us later on to 

We start with two simple lemmas giving the behaviour of $\psi$ at respectively the macroscopic scale and at large microscopic scale.

\begin{lemma}\label{lem:global_psi}
The global shift in the definition of $\psi$ can be chosen so that $\psi - \phi = O(\delta^{1/3})$.
\end{lemma}
\begin{proof}
This is a computation. First let us consider the case of two points aligned vertically. We get (removing $\delta$ from the notation of $M$ and $M^{\pm}(w_i) = M^{\pm}_i$ for ease of notation)
\begin{align*}
\psi (z + k e_2) - \psi(x) & = \sum_{i = 1}^k  \delta  \Big( \left( \lambda F( w_i) + \bar \lambda \bar F (w_i) \right)\bar \lambda G (b_i) + 0(\delta^{29}) \Big)\\
	& = \delta \sum_{i=1}^k d_y \phi (w_i) \left( 1 + \delta( M^+ _i - M^-_i) - \delta^2 M^+_i M^-_i  \right) +  O(k\delta^{30}) 
		 + \delta \sum_{i=1}^k \bar F (w_i) G(b_i) \\
	& = \phi(z + ky) - \phi(z) + O(k\delta^2) + \delta \bar \lambda^2 \sum_{i=1}^k \bar F(w_i) G(b_i).
\end{align*}
The term $O(k \delta^2)$ is at most of order $O(\delta)$ so we only have to focus on the sum. Writing it explicitly, we have
\begin{align*}
\bar F G & = e^{-2 i \Im( H)/\delta} | d_y \phi| \Big(1 +\delta (\bar M^+ - M^-) - \delta^2 \bar M^+ M^- \Big) \\
\end{align*}
so all the terms with at least one power of $\delta$ give contribution of the right order. Now
\[
e^{-2 i \Im( H)/\delta} | d_y \phi| (z + k e_2) = \exp\Big( -2 i\Im(H(z))/\delta -2 i d_y \Im(H(z)) k + O(\delta k^2) \Big)\Big( |d_y \phi| + O(\delta k)  \Big)
\]
so, using that $d_y \Im(H(z)) = i \pi (p_a - 
1) \neq 0$, we get
\[
\delta \sum \bar F G = \delta e^{-2 i \Im H(z)/\delta} | d_y \phi|\Big( \sum_k e^{-2 i d_y \Im(H(z)) k  } +  0(\delta k^3 ) \Big) = \delta ( 0(1) + O(\delta k^3) ).
\] We see that the sum over $k = \delta^{-1/3}$ terms is of order $\delta$ which gives the an error term $O(\delta^{1/3})$ over a macroscopic distance of order $\delta^{-1}$ as desired.

For an increment in another direction the proof is similar, the computations are just a bit more complicated. 
\end{proof}

It is useful to extend the value of $\psi$ linearly along the edges of $\cH^*$ so that we can think of $\psi( \cH^*)$ as a closed connected set.

\begin{notation}
For all black vertex $b$, we call $w_1, w_2, w_3$ its adjacent white vertices in positive order and starting from the one immediately to its left (meaning $w_1$ and $b$ have the same coordinates). We let $\Delta(b)$ denote the triangle with vertices $0$, $F(w_1) |G(b)| $ and $- F(w_3) |G(b)|$. We let $\Delta (w)$ be defined similarly but for white vertices.
We also write with a slight abuse of notation $\psi(w)$ or $\psi(b)$ for the image of a white or black face of $\cH^*$. 
\end{notation}

\begin{lemma}\label{lem:meso_psi}
For all $v$ in the dual of $U_{ex}^\d$, let $w$ be the white vertex immediately above $v$ and let $T(v)$ be the whole plane T-graph with parameters $\Delta (v) = (0, d_y \phi(w)i, d_y \phi(w) \Phi(w) )$ and $\lambda(v) =\lambda e^{i \Im (H(w))/\delta} \sqrt{ \frac{d_y\phi}{|d_y \phi|}}$ obtained by choosing $v$ as base point in the construction. Assume for ease of notation that $\psi(v) = T(v) = 0$. Using the convention that $\psi$ is extended along edges, we have
\[
\forall M, d_H\Big( \frac{1}{\delta} \psi \cap B(0, M) ) , T(v) \cap B(0, M) \Big) = O( \delta M^{3}),
\]
where $d_H$ is the Hausdorff distance.
\end{lemma}
\begin{proof}
We start by comparing the flow $\Omega$ with the flow $\Omega_T$ used in the construction of $T$, assuming without loss of generality that $v$ and $w$ have coordinates $(0, 0)$. It is not hard to check by a direct Taylor expansion that for all $m,n$,
\[
\Omega( b(m,n) w(m,n) ) = \delta  (\partial_y \phi(w) ) \Omega_T ( b(m,n) w(m,n) ) \big(1+ O((m^2+n^2)\delta) \big).
\]
Summing over a path, we see that $\psi(m,n) - \delta  \partial_y \phi(w)  T(m,n) = O(( |m|^3+|n|^3)\delta )$.
Since $T$ is a usual full plane T-graph, we know that the map $T$ given by a linear map plus a bounded correction. In particular the edges of $T \cap B(0, M)$ have $O(M)$ coordinates which concludes the proof.
\end{proof}

\begin{remark}
The exact expression of the parameters of the T-graph $T(w)$ depends on all the choices made above which explains why it is a bit heavy, however it is not really important in the following : the point is that the map $\psi$ is well approximated locally by \emph{some} T-graph.
\end{remark}

\begin{corollary}\label{cor:Deltaw}
For any $w$ and $b$,
\begin{gather}
\Delta(w) = \Big(0, \partial_y \phi(w) e^{- i \Im(H(w))}, \partial_y \phi(w) \Phi(w)e^{- i \Im(H(w)) }\Big)+ O(\delta) \\
\Delta (b) = \Big(0, \partial_y \phi(b) e^{i \Im(H(b))}, \partial_y \phi(b) \Phi(b)e^{i \Im(H(b)) } \Big)+0(\delta),
\end{gather}
In particular we have uniform upper and lower bounds on the size of $\Delta(w)$, $\Delta(b)$ and we have a uniform lower bound on the angles.
\end{corollary}

\begin{remark}
The error term in \cref{lem:meso_psi} probably has an optimal power in $\delta$. Indeed we expect the angle of each triangle to be close to the occupation probability of the corresponding edges which should change smoothly like an approximation of $\nabla h_\infty$. On the other hand in the whole plane T-graph $T(w)$, all triangles have the same aspect ratio. Therefore the best match we can expect is $O(\delta)$, corresponding to the fact that this is a form of approximation of a curved surface by a plane. 
\end{remark}

Now we turn to the behaviour at the microscopic scale. Note that \cref{lem:meso_psi} already offers some control at that scale but we will more precision, as suggested by the order $30$ error we used in \cref{lem:def_FG}.

\begin{lemma}\label{lem:psiw}
For every $w$ and $b$, looking up to translation, we have
\begin{gather}
\psi(w) = 2 \delta \bar \lambda \Re( \lambda F(w)/|F(w)| ) \Delta(w) + O(\delta^{29}) \\
\psi(b) = 2 \delta \bar \lambda \Re( \lambda \Delta(b) ) G(b)/|G(b)| +O(\delta^{29}),
\end{gather}
where the $0(\delta^{29})$ is uniform over $\delta$, $w$, $b$ and we wrote with the slight abuse of notation $\psi(w)$ for the image of the white triangle corresponding to $w$, extended linearly over the edges and similarly for $b$. By the equality up to $O(\delta^{29})$ between triangles, we mean that the vertices lie at distance $O(\delta^{29})$, which also corresponds to an Hausdorff distance.
\end{lemma}
\begin{proof}
This is a direct consequence of the definition : The equation holds without any error term is the paths chosen to define $\psi$ go through the edges $b w_1^*$ and $b w_3^*$ last. The results therefore holds in the general case because the choice of path only contributes $0(\delta^{29})$.
\end{proof}

\begin{lemma}\label{lem:choice_lambda}
For all $\delta$ we can find $\lambda$ so that 
\[
\forall w, \text{Diam}( \psi(w) ) \geq \Theta( \delta^3).
\]
For the rest of the paper we always assume that $\lambda$ is chosen in such a way.
\end{lemma}
\begin{proof}
As mentioned in \cref{cor:Deltaw}, the sizes of the triangles $\Delta(w)$ are uniformly bounded away from $0$. By \cref{lem:psiw} we only have to show that we can choose $\lambda$ so that $\Re( \lambda F(w) /|F(w)|) \ge \Theta( \delta^2)$ for all $w$. Since the points $F(w)/|F(w)|$ do not depend on $\lambda$ and there are $O(\delta^{-2})$ such points, we can find at least one way to rotate them so that no point is within $O(\delta^2)$ of the imaginary axis which concludes the proof.
\end{proof}

Now we are in a good shape to reproduce the geometric results of \cref{sec:general_Tgraph} because we know that the smallest geometric features are expected to live at the scale $\delta^3$ which is much bigger than the errors in $\psi$ (see \cref{fig:psi} for a schematic representation of the image of $\psi$).

\begin{corollary}\label{cor:flat_triangle}
For any $b$, $\psi(b)$ is a flat triangle in the following sense. It has diameter $O(\delta)$, a smallest edge of size at least $O(\delta^3)$ but it's inscribed circle has diameter $O( \delta^{29} )$. As a consequence exactly one of it's vertex is at distance $O(\delta^{29})$ of the opposite edge and we say that this is the \emph{almost interior} vertex of $\psi(b)$. The other points are called extremal points.
\end{corollary}
\begin{proof}
This is a direct consequence of saying that $\psi(b) = 2 \bar \lambda \Re( \lambda \Delta(b) ) G(b) +o(\delta^{29})$ and our choice of $\lambda$ so that $Re(\lambda F(w) ) \geq c \delta^3  $.
\end{proof}

\begin{figure}[t]
\begin{center}
\includegraphics[width=.6\textwidth]{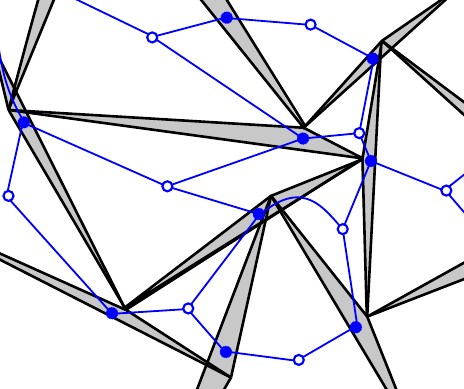}
\caption{A schematic representation of the map $\psi$. The solid edges are the images of the dual edges and the shaded regions are the images of black dual faces. The blue vertices and edges show the underlying  heavily distorted hexagonal lattice.}\label{fig:psi}
\end{center}
\end{figure}

\begin{lemma}\label{lem:overlap}
The measure of the set of points located inside at least two white faces is at most $O(\delta^{27})$.
\end{lemma}
\begin{proof}

By \cref{lem:psiw,lem:choice_lambda}, the images of all white faces have the same orientations. Furthermore by \cref{cor:flat_triangle}, the image of any black vertex has area $O( \delta^{29})$.

Suppose by contradiction that the white faces are overlapping in an area more than $\Theta(\delta^{27})$, then one can find a point in the intersection which does not belong to any black face. The path following the two overlapping faces has winding $4\pi$ around the point. On the other hand, one can find a path enclosing both faces and winding once around the point by taking a large enough path and using \cref{lem:global_psi}. This is a contradiction because we can write the winding of the big path as a sum of winding around every face inside : black faces have no contribution since we took the point outside of all of them while white faces have a positive contribution by \cref{lem:psiw} and \cref{lem:choice_lambda}.
\end{proof}

\begin{lemma}\label{lem:Tintersection}
For every dual vertex $v$, $\psi(v)$ is the almost interior point of exactly one $\psi(b)$ and an extremal point of the two others adjacent black flat triangles. Furthermore the three angles between the flat triangles adjacent to $\psi(v)$ are bounded away from $0$ and $\pi$. 
\end{lemma}
\begin{proof}
Fix a dual vertex $v$ and let us denote the adjacent black and white vertices by $b_1$, $w_1$, $b_2$, $w_2$, $b_3$, $w_3$ in clockwise order and let $v_1$, $v_6$ be the adjacent dual vertices in clockwise order with $v_1$ across from $b_1, w_1$.

Since each black face is a flat triangle, the quantity
\[
A =\arg \prod_i \frac{\psi( v_{2i+1}) - \psi(v)}{\psi( v_{2i}) - \psi(v)} 
\]
is either $0$ if $\psi(v)$ is an extremal point of an even number of flat triangles or $\pi$ if it is an extremal point of an odd number of flat triangles, up to an error $O(\delta^{26})$ (the magnitude of the angles in a flat triangle). On the other hand, writing it explicitly, we obtain
\[
A = \arg \prod_i \frac{2 \delta \bar \lambda \Re( \lambda F(w_{i+1})) G(b_i) K(w_{i+1},b_i) + O(\delta^{29}) }{2 \delta \bar \lambda \Re( \lambda F(w_{i-1})) G(b_i) K(w_{i-1},b_{i})+O(\delta^{29})}
\]
which gives after simplifying the $F$ and $G$ using the fact that $\Re( \lambda F(w_{i+1})) \geq \Theta(\delta^3)$ :
\[
A = \arg \prod \frac{K(w_{i+1},b_i)}{K(w_{i-1},b_{i})} + O(\delta^{26}) =  O(\delta^{26})
\]
so $v$ has to be an extremal point of either zero or two adjacent flat triangles.
It is then easy to see from the preservation of orientation that if $v$ is an interior point of three flat triangles then there must be a significant overlap (see \cref{fig:image_locale}) which concludes the proof.

\end{proof}

\begin{figure}
	\begin{center}
		\includegraphics[width=.8\textwidth]{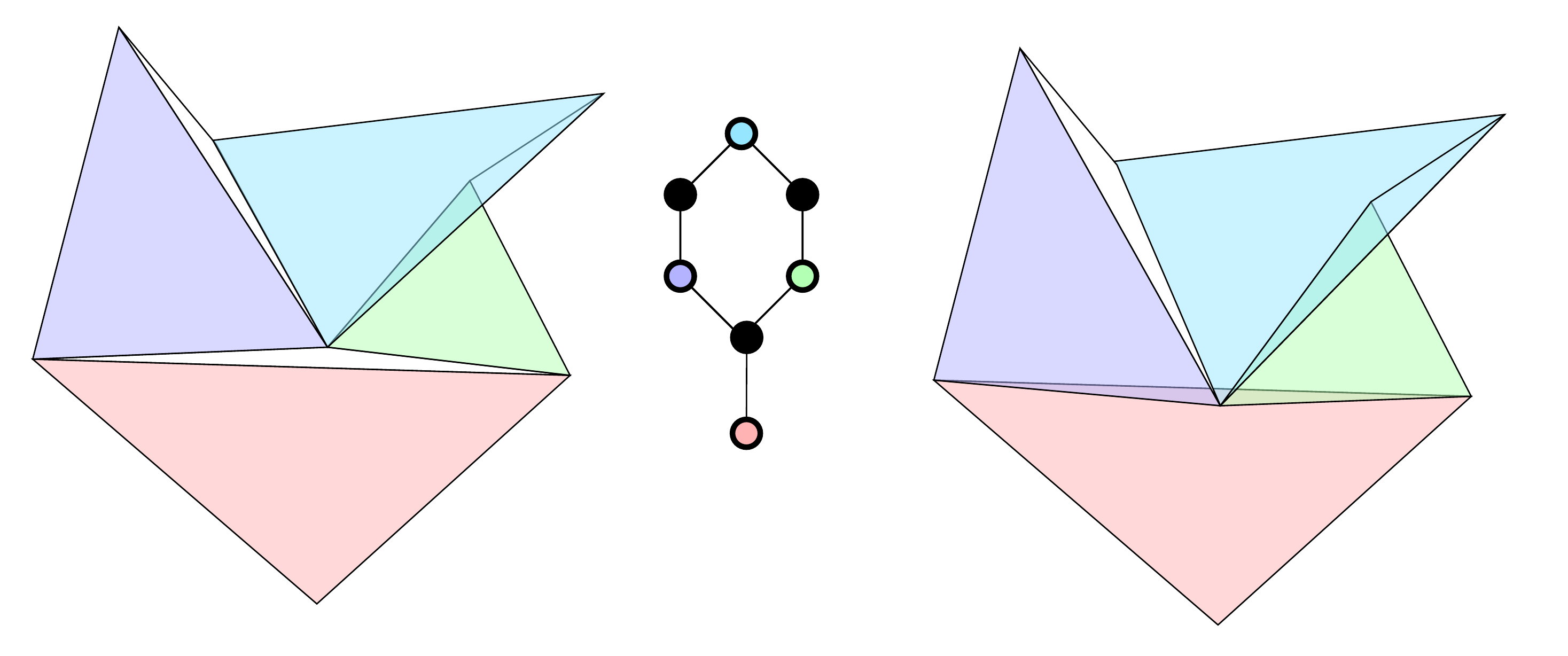}
		\caption{Schematic representation of the two possible local geometries around a vertex in $\psi$. The piece of the hexagonal lattice in the centre shows which vertices are drawn. The colour are matched to indicate the image of the white faces in the left and right part, while the image of black vertices are left empty for legibility.}\label{fig:image_locale}
	\end{center}
\end{figure}

\begin{lemma}\label{lem:distance_vertices}
For any dual vertex $v$ away from the boundary the following holds. Let $b_1, b_2, b_3$ be the three black vertices adjacent to $v$, then for all black vertex $b$ other than $b_1, b_2, b_3$,
\[
B\Big( \psi( v ) , \Theta( \delta^3) ) \Big) \cap \psi( b) = \emptyset. 
\]
In particular, if $v \neq v'$ are two dual vertices, then $|\psi( v) - \psi( v')| \geq \Theta( \delta^3) $.
\end{lemma}
\begin{proof}
Fix $v$ away from the boundary. Let $w_1, w_2, w_3, b_1, b_2, b_3$ be the neighbouring white and black vertices of $v$. By \cref{lem:Tintersection}, $\psi(v)$ is an almost interior point  of exactly one of the $b_i$ which we assume by symmetry to be $b_1$.
 Let $w_4$ be the last white neighbour of $b_1$. By the preservation of orientation, \cref{lem:overlap} and the $\Theta( \delta^3 )$ bound on the length of edges from \cref{lem:choice_lambda}, we see that the triangles $\psi(w_i)$ must cover a ball $B\Big( \psi(v), \Theta( \delta^3 )\Big)$ up to at most $O(\delta^{27})$ area.

 On the other hand, the same argument seen now around $\psi(b_1)$, together with the fact that the angles of the $\psi( w)$ are bounded away from $0$ and $\pi$ by \cref{lem:psiw} shows that the $\psi( w_i)$ also cover a neighbourhood of $\psi( b_1)$ up to an $O( \delta^{27} )$ area. To be more precise, we can find a strip of width $\Theta( \delta^4 )$ and starting and ending at distance $\Theta( \delta^4 )$ from the two extremal points of $\psi( b_1 )$ almost fully covered by the $\psi( w_i)$. This clearly holds for all black vertices away from the boundary.

Now if we assume by contradiction that a point $x \in \psi( b)$ intersects, $B(\psi(v), \frac{1}{2}\Theta(\delta^3 ) )$, then the intersection of $B(\psi(v), \Theta(\delta^3 ) )$ with the neighbourhood of $\psi( b)$ defined in the previous paragraph must be at least $\Theta( \delta^4 \delta^3 )$ which is a contradiction with \cref{lem:overlap}.
\end{proof}

\subsection{Correction into a T-graph}\label{sec:correction}

In this section, we correct $\psi$ into an actual T-graph. The construction is straightforward : first we just replace each black triangle by segment in an arbitrary way, then we shorten all segments equally to get rid of any possible intersection, finally we regrow the segments one by one to obtain a T-graph. 
\begin{definition}
For any $b$, let $S_1 (b)$ be the closed segment between the two extremal points of $\psi(b)$ and let $\Gamma_1 = \cup_b S_1(b)$. Let $S_2(b)$ be defined by removing a length $\delta^{13}$ from each side of $S_1(b)$ and let $\Gamma_2 = \cup_b S_2(b)$.
\end{definition}

\begin{lemma}
If $\delta$ is small enough then all the segments $S_2(b)$ in $\Gamma_2$ are disjoint.
\end{lemma}
\begin{proof}
Suppose that this is not the case, then there exists $b$, $b'$ such that $\psi(b)$ and $\psi(b')$ intersect at distance more than $\delta^{13}$ of their extremal points. We see in \cref{fig:image_locale} that around any such point we can find an area of measure at least $\theta(\delta^{26})$ covered by the neighbouring white faces. Since only a single white face can be adjacent to both $b$ and $b'$, we can see that there must remain an intersection of size at least $\theta(\delta^{26})$ between two different white faces which is a contradiction with \cref{lem:overlap}.
\end{proof}

\begin{definition}\label{def:gammaex}
Now we fix arbitrarily an order on all interior black vertices of $U_{ex}^\d$. Looking at the vertices $b_i$ in order we define open segments $S(b_i)$ by growing each end of $S_2(b_i)$ until it hits either $\Gamma_2$ or an $S(b_j)$ for $j < i$. We let $\Gamma_{ex} = \cup S(b_i)$. 
\end{definition}

This is the graph mentioned in the beginning of this section so \cref{def:gammaex} concludes the construction part of the proof. Now we turn to the analysis of the basic properties of $\Gamma_{ex}$, namely proving that it is indeed a T-graph and that its associated dimer graph is close to the hexagonal lattice. 
The reader might want to look at the bottom of \cref{fig:contribution_flot} for an exhaustive list of the possible local geometries of an intersection since the proof essentially boil down to stating that no other geometry is possible. 

\begin{lemma}\label{lem:isolated_vertices}
For all $\delta$ small enough, for all $b$, the length added to $S_2(b)$ to obtain $S(b)$ is at most $\Theta(\delta^{13})$ on each side.

Furthermore, if $b, v$ are respectively a black vertex and a face of $U_{ex}^\d$ such that $\psi(v)$ is an extremal point of $\psi(b)$, then the endpoint of $S(b)$ close to $\psi(v)$ is in the segment associated to one of the two other black vertices around $v$.
\end{lemma}
\begin{proof}
Fix $b, v$ a black vertex and dual vertex such that $\psi(v)$ is an extremal point of $\psi(b)$ and let us show that less than $O(\delta^{13})$ is added on the side of $\psi(v)$.

By \cref{lem:Tintersection}, $\psi(v)$ is an interior point of $\psi(b')$ with $b'$ another of the neighbour of $v$ and the angle between $S_1(b)$ and $S_1(b')$ is bounded away from $0$ or $\pi$. Furthermore $\psi(b')$ is a flat triangle so $\psi(b')$ is at distance at most $O(\delta^{29})$ of $S_1(b')$ and therefore also of $S_2(b')$ since the interior points are $0(\delta^3)$ away from the endpoints. Overall we see that the endpoint of $S_2(b)$ is at distance at most $\delta^{13} + O(\delta^{29})$ of $S_2(b')$ with a non-parallel direction so adding $C\delta^{13}$ to $S_2(b)$ would create an intersection with $S_2(b')$, with $C$ only depending on the bound on the angle. Finally hitting a segment other than $S_2(b')$ can only further reduce the length added.

Finally, for the second part of the statement, we notice that by \cref{lem:distance_vertices} segment other than $S_2(b)$, $S_2(b')$ and the one associated with the last black neighbour of $v$ cannot have a point within $O( \delta^{13})$ of $\psi(v)$ while we know that the intersection will happen in this region.
\end{proof}

\begin{proposition}\label{prop:segments}
Away from the boundary, one can divide edges of $\Gamma_{ex}$ into two types, small edges with length $O(\delta^{13})$ and long edges with length at least $\Theta (\delta^{3} )$. Each segment of $\Gamma_{ex}$ contains exactly two long edges and up to three small edges no two of which can be adjacent. In other word, a segment is made of two long edges with may be separated from each other or from the endpoint by a small edge.
\end{proposition}
\begin{proof}
Fix a segment $S(b)$ and let $v_1, v_2, v_3$ be the three dual vertices adjacent to $b$. By \cref{cor:flat_triangle}, assume that $\psi(v_3)$ is the almost interior point of $\psi(b)$. By \cref{lem:isolated_vertices} all vertices of $\Gamma_ex$ along $S(b)$ are within $O( \delta^{13})$ of one of the $\psi(v_i)$ and are intersections with a segment corresponding to a black vertex adjacent to $v_i$. In particular all vertices of $\Gamma$ along $S(b)$ must be in one of three regions of size $O(\delta^{13})$ which are separated by at least $\Theta( \delta^3)$ which proves that each segment contains at most two long edges. 

On the other hand, near $\psi( v_3)$, exactly two other segments have an endpoint by \cref{lem:Tintersection} and \cref{lem:isolated_vertices}, therefore $S(b)$ must have either one or two vertex near $\psi(v_3)$. Similarly near $\psi( v_1)$ and $\psi( v_2)$, exactly one other segment has an endpoint and therefore $S_(b)$ can have either $0$ or a single vertex in its interior. This concludes the proof.
\end{proof}

\begin{definition}
We let $\tilde{U}_{ex}^\d$ be the dimer graph associated with $\Gamma_{ex} $ as per \cref{def:dimers}. Recall that white vertices of $\tilde{U}_{ex}^\d$ are the faces of $\Gamma_{ex}$ and that the black vertices are the segments of $\Gamma_{ex}$ with the natural adjacency relation, with Kasteleyn matrix given by the common parts between vertices and faces. We let $\tilde w$ denote the corresponding weight.
\end{definition}

\begin{proposition}\label{prop:geometrie}
Away from the boundary, $\tilde U_{ex}$ and $U_{ex}^\d$ have the following relation.
\begin{itemize}
\item One can identify their vertex set.
\item $U_{ex}^\d$ is the subgraph generated by all vertices of $\tilde U_{ex}$ and the duals of all long edges of $\Gamma_{ex}$.
\item Edges of $\tilde U_{ex}$ dual to small edges connect opposite vertices inside a face of $U_{ex}^\d$.
\item Each face of $U_{ex}^\d$ corresponds to either one or two faces of $\tilde U_{ex}$ and in the later case the corresponding vertices are connected by a small edge in $\Gamma_{ex}$.
\end{itemize}
When we say that a relation between two graphs holds away from the boundary, we mean that there are subgraphs of $\tilde U_{ex}$ and $U_{ex}^\d$ which satisfy the above relations and only differ from the full graphs on a $O(\delta)$ neighbourhood of the boundary.
\end{proposition}
\begin{proof}
The identification of black vertices is trivial since they are naturally associated with segments. For white vertices, by \cref{lem:distance_vertices} we see that two segments cannot be almost parallel and at distance smaller that $\Theta( \delta^4 )$. Furthermore, since there cannot be three vertices at distance smaller than $\Theta( \delta^{13} )$, any face of $\Gamma_{ex}$ must have at least a side of length $\Theta( \delta^3 )$. Overall we see that any face has area at least $\Theta( \delta^7 )$. Since the $\psi( w)$ have small overlap by \cref{lem:overlap}, there must exists a single $\psi(w)$ covering any face and we can associate therefore associate each face to a single white vertex.
The identification of edges of $U_{ex}^\d$ to long edges is then straightforward.

Now since $\tilde U_{ex}$ is still a planar graph and has the hexagonal lattice as a subgraph, the only possibility is to add edges cutting an hexagon in two which proves the last two items.
\end{proof}

\begin{proposition}\label{prop:comp_K}
For any face $v$ of $U_{ex}(\delta)$ away from the boundary with vertices $b_i, w_i$,
\[
1 = \prod_i \frac{w(b_i, w_i)}{w(b_i, w_{i+1})} = \prod_i \frac{\tilde w(b_i, w_i)}{\tilde w(b_i, w_{i+1})} + O(\delta^{10}).
\]
\end{proposition}
\begin{proof}
Fix a face $v$ and let $b_i$, $w_i$ be the adjacent vertices. By \cref{lem:isolated_vertices}, the length of any edge $(bw)$ is modified by at most $O(\delta^{13})$ between $\psi(U_{ex})$ and $\Gamma$ and the difference between increments of $\psi$ and $\Omega$ is at most $O(\delta^{29} )$ so
\[
\prod_i \frac{\tilde w(b_i, w_i)}{\tilde w(b_i, w_{i+1})} = \Big| \prod \frac{\Omega(b_i w_i) + O(\delta^8)}{\Omega(b_i w_{i+1})+ O(\delta^{13})} \Big|
\]
Now by \cref{lem:choice_lambda}, the flow $\Omega$ is at least $c\delta^3$, therefore
\[
\prod_i \frac{\tilde w(b_i, w_i)}{\tilde w(b_i, w_{i+1})} = \Big| \prod \frac{2 \delta \bar \lambda \Re( \lambda F(w_i)) G(b_i) w(w_i,b_i)}{2 \delta \bar \lambda \Re( \lambda F(w_{i+1})) G(b_i) w(w_{i+1},b_i)}(1 + 0(\delta^{10})) \Big|
\]
which simplifies into the desired expression.
\end{proof}

\subsection{Flow and height function on $U^\d_{ex}$.}\label{sec:flow}

Since $\Gamma_{ex}$ is a T-graph, at least away from the boundary, there is a canonical way to compute height functions on $\tilde U_{ex}$ which goes back to Kenyon \cite{Kenyon2007}.

\begin{definition}
The so called reference flow $M_{ref}$ on $\tilde U_{ex}$ is defined (for $w$ and $b$ a white and a black vertex of $\tilde U_{ex}$) as follows. If $w$ and $b$ are not adjacent, then $M_{ref} (wb) = 0$. Otherwise, let $[x_1,x_2]$ be the intersection between $\Gamma_{ex}^\d(w)$ and $\Gamma_{ex}^\d(b)$ and let $S_1,S_2$ be the two open segments such that $x_1 \in S_1$ and $S_2 \in S_2$. We set $M_{ref} (bw)$ to be the sum of the angles away from $\Gamma_{ex}^\d(w)$ at $x_1$ (resp. $x_2$) between $S_1$ (resp. $S_2$) and $\Gamma_{ex}^\d(b)$ divided by $2\pi$. See \cref{fig:def_flow} for the geometric picture which also explains simply why this is a correct reference flow for defining a dimer height function.
\end{definition}

\begin{figure}
\begin{center}
\includegraphics[width=.5\textwidth]{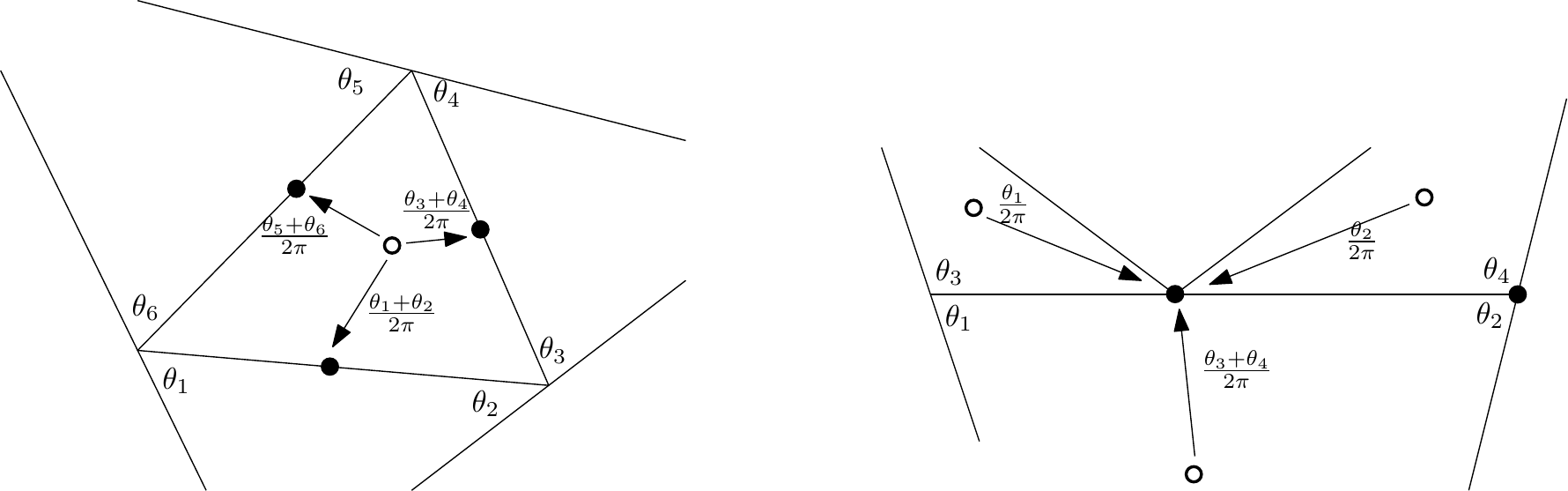}
\caption{The definition of the reference flow on a general T-graph}\label{fig:def_flow}
\end{center}
\end{figure}

Naturally, we need to compare this flow to our target asymptotic height function $h^\cC$.

\begin{lemma}\label{lem:height_macro}
Let $h_{ref}$ be the height of the flow $M_{ref}$ compared to the standard flow sending $1/3$ from each white to black (in the sense that $h_{ref}$ is the primitive of the difference of their dual). We have
\[
h_{ref} = \frac{1}{\delta} h^\cC + 0(1).
\]
\end{lemma}
\begin{proof}
Fix a face $v$ in $U_{ex}^\d$ and $n$ such that $\{ v + \delta i e_2 | i \leq n \}$ stays away from the boundary. For ease of notation we write $v_i = v + \delta i e_2$ and we call $\eta$ the oriented path from $f_1$ to $f_n$. Even though $\eta$ is a path on the hexagonal lattice, it can also be seen as a path on $\tilde U_{ex}$, up to two choices at the starting and ending points. Up to an $O(1)$ error due to these choices, we can therefore speak of the total flow of $M_{ref}$ across $\eta$ which we will write as $M_{ref}(\eta)$.

It is easy to see that $\Gamma_{ex}(\eta)$ is a non self crossing path with exactly one segment for each edge of $\tilde U_{ex}$ crossed by $\eta$. Also, we note that the contribution to $M_{ref}(\eta)$ of each segment can be written as the sum of two angles, one at each endpoint of the associated piece of $\Gamma_{ex}(\eta)$, up to an $O(1)$ error for the starting and ending points, we can therefore rewrite $M_{ref}(\eta)$ as a sum over all angles in $M_{ref}(\eta)$. The contribution of each angle to $M_{ref}$ are easy (if somewhat cumbersome) to derive and are given for all possible geometry in \cref{fig:contribution_flot}
Comparing these terms with the windings of the path $\Gamma_{ex}(\eta)$, we see that
\begin{equation}
W(\eta) + 2 \pi M_{ref}(\eta) = \sum_{k} 2 \gamma_k  + O( 1) 
\end{equation}
where $\gamma_k$ is the angle in the white face between $\Gamma_{ex}(b_n)$ and $\Gamma_{ex}(b_{n+1})$ called $\gamma$ in \cref{fig:contribution_flot}. The $O(1)$ term comes from the starting and ending point. Note that by \cref{lem:psiw} and since we measure this angle positively by convention, we have $\gamma_k = \pi p_a( b_k) + O(\delta)$. Recalling that $p_a$ is related to the vertical derivative of $h^\cC$, we get overall
\begin{equation}\label{eq:reference_height}
h_{ref} - \frac{1}{\delta} h^\cC = \frac{-1}{2\pi} W(\eta) + O(1).
\end{equation}

Now we need to argue that the winding term stays bounded. For this we first discuss the basic case of a T-graph in the family $T_{\Delta, \lambda}$ for a non-extremal $\Delta$. In that case, the map from the hexagonal lattice to the T-graph is given by a non-degenerate linear term  plus a bounded correction (see \cref{prop:geometryTgraph}) and therefore the image of any straight line will have a uniformly bounded winding (using the fact that it is still a non self crossing curve). Furthermore the image will itself be given  at any scale large enough compared to $\delta$ by a straight line whose direction depends smoothly on $\Delta$.

By \cref{lem:meso_psi}, the above argument also applies up to distance $\Theta (\delta^{1/3})$ and in particular shows that any vertical path of length $o(\delta^{1/3} )$ will have bounded winding and will be given by a straight line up to an $O(\delta)$ error. Now consider a path $\eta$ of length $o(1)$ and divide it into pieces $\eta_i$ of length $o(\delta^{1/3})$ but much larger that $\delta$. Each of the $\eta_i$ is roughly a straight line up to an $O(\delta)$ error and all directions are given by smooth functions of $\Phi$. Since $\Phi$ itself is smooth, all the directions of the $\Gamma(\eta_i)$ are close to each other. In particular, by taking the length of the first and last piece large enough we can make sure that all the other pieces stays inside an infinite strip which contains neither the starting nor the ending points of $\Gamma_{ex}(\eta)$. Overall, the winding is still bounded in that case.

For a path of arbitrary length, we note that the previous argument actually holds for paths of length up to some $\epsilon$ depending only on the map $\Phi$. The argument therefore extends to all path just by cutting any path into segments of length $\epsilon$ and bounding the total winding by the sum over all pieces.

The above argument showed that the winding term is uniformly bounded so we can  We can therefore compute $h_{ref}$ directly as
\begin{align*}
h_{ref} (f_n) - h_{ref} (f_1) &=  M_{ref} - n/3 \\
  & = O(1) + \sum_{k} (p_a( b_k) - 1/3) \\
& = O(1) + \frac{1}{\delta}\int_{f_1}^{f_n} \partial_y h^\cC .
\end{align*}
where in the last line we used that since $\partial_y h^\cC$ is smooth, the approximation of the integral by a Riemann sum is accurate up to $O(1)$.

Finally the arguments given for a vertical path extend to linear paths in the two other lattice directions, and we conclude since we can connect any pair of points  in $U^\d$ by a bounded number of segments staying in $U_{ex}^\d$ and away from the boundary. This concludes the proof.
\end{proof}

\begin{figure}
\begin{center}
\includegraphics[width=.8\textwidth]{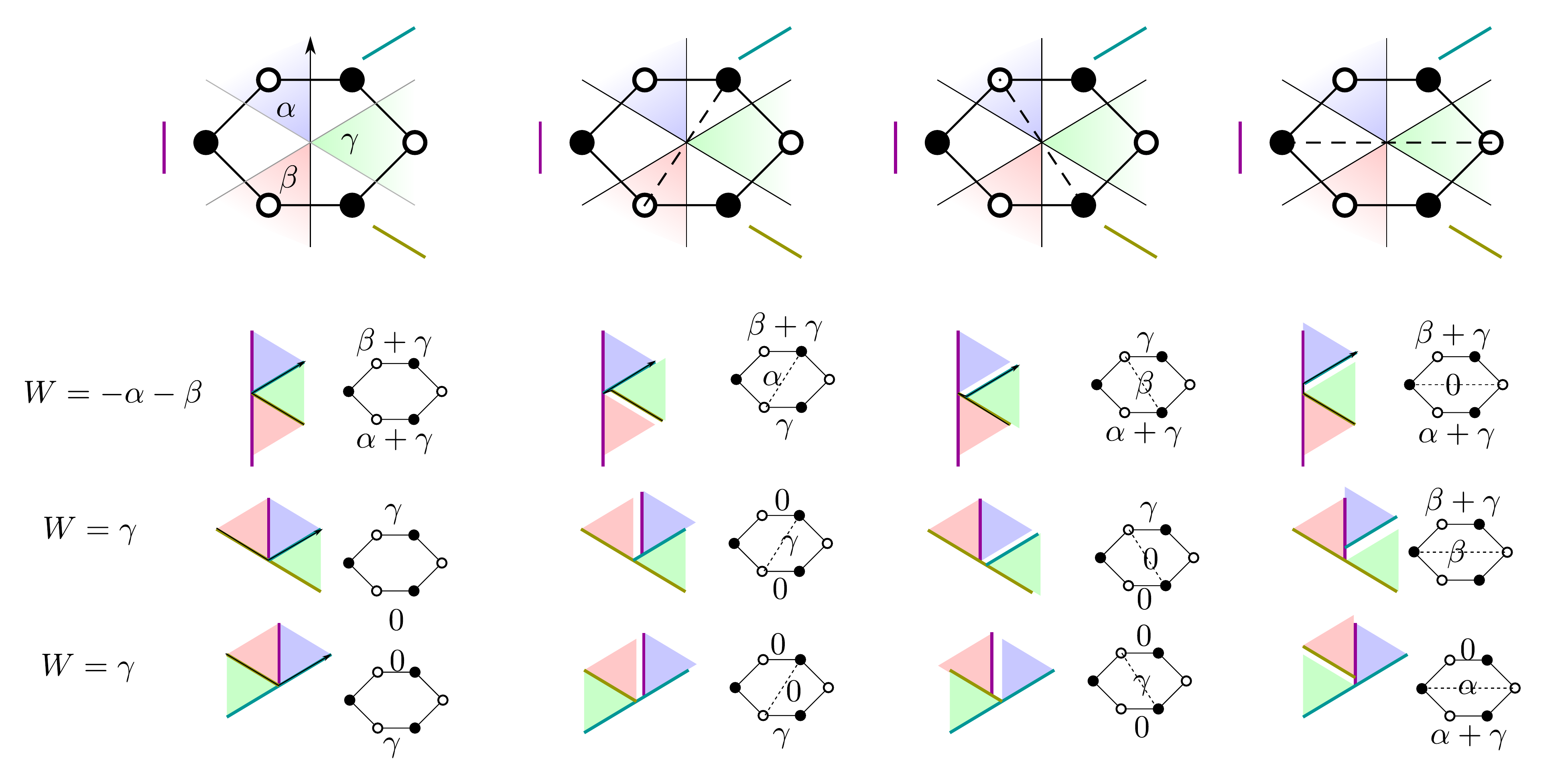}
\caption{The reference flow and winding for all possible geometries in the image of a face of $U_{ex}^\d$. The top row displays the $4$ possibilities for $\tilde U_{ex}$ and there are no others since $\tilde U_{ex}^\d$ must be a simple bipartite planar graph containing the original face. The three rows below give the possible local geometries in $\Gamma_{ex}$ depending on which of the three segments has an interior point and the associated reference flow for the relevant edges. The map $\Gamma_{ex}$ is pictured by using the same colour, i.e. the image of the white vertex with a blue background is the blue region and so on. The angles $\alpha, \beta, \gamma$ are defined as angles in $\Gamma_{ex}$ even though they are pictured in the first row for convenience.
	Note that in the last column, the dotted line is crossed with the white vertex of the right and therefore has a negative contribution compared to the rest.}
\label{fig:contribution_flot}
\end{center}
\end{figure}

\subsection{The graphs $U(\delta)$ and $\Gamma(\delta)$}\label{sec:graph}

At this stage, we have defined a ``dimer graph'' $\tilde U_{ex}^\d$ and it's associated T-graph $\Gamma_{ex}$ but without even attempting to be precise near the boundary. In particular there is no reason why $\tilde U_{ex}^\d$ should have a boundary height function close to $h^\cC$. As mentioned above, the idea is to cut a piece $\Gamma \subset \Gamma_{ex}$ by a loop-erased random walk path to insure that the subgraph $\Gamma$ is compatible with the correspondence between dimers and uniform spanning tree. This is completely similar to the strategy in the flat case given in \cite{BLRannex}, but there is an additional complication here stemming from the fact that we want to ultimately consider dimers on the hexagonal lattice and not on $\tilde U(\delta)$.

\begin{definition}\label{def:standard}
In a random walk path on $\Gamma_{ex}$, we call moves along long edges ``long jump'' and moves along small edges ``small jumps''. We say that a random walk path $X$ on $\Gamma_{ex}$ is \emph{standard} if is satisfies the following. Whenever $X$ does a long jump, if it arrives at a point where at small jump is possible then the next move of $X$ must be along this small jump (if it arrives at a position where only two long jumps are possible there is no condition. The same holds also for the first move of $X$ and for each move after a time where $X$ goes from a segment to another segment.
\end{definition}

\begin{lemma}\label{lem:standard}
Let $\tau$ denote the minimum of $\delta^{-3}$ and the first time the random walk comes within $\Theta( \delta)$ of the boundary of $\Gamma_{ex}$, with a slight abuse of notation we call $X[ 0, \tau]$ the random walk on $\Gamma_{ex}$ up to time $\tau$. We have
\[
\P_v( X[ 0, \tau] \text{ is standard}) \geq 1 - O(\delta^{4} ).
\]
\end{lemma}
\begin{proof}
Recall that by \cref{prop:geometrie}, the minimal length of a long jump is $\Theta(\delta^3)$ and that jumps happen at poissonian times. Therefore by \cref{lem:exit}, up to time $\tau$, the walk has an exponentially small probability to make more than $O(\delta^{-6})$ long jumps. Further note that the random walk on $\Gamma_{ex}$ induces a walk on the faces of the hexagonal lattice and a walk on segments or black vertices by the identification of \cref{prop:geometrie}. Since the walk on dual vertices can only move with  long jumps, it also makes at most $O(\delta^{-6})$ steps. Finally each dual vertex can have images inside at most two segments so the walk on segments also makes at most $O(\delta^{-6})$ steps. 

Recall that in \cref{prop:geometrie} we said that each segment contains exactly two long pieces and up to three small pieces no two of which are adjacent. In particular from any vertex in $\Gamma_{ex}$, the walk has either two possible long jumps or a long and a short one. in the later case, the probability that the first jump is long is bounded by $O( \delta^{10} )$. By union bound and the bound on the number of steps from the previous paragraph, we see that with probability $1-O(\delta^{4})$, the walk is standard.
\end{proof}

Imputing this result in the uniform crossing (\cref{lem:uniform_crossing}), we immediately obtain the existence of standard paths realising microscopic crossings.
\begin{corollary}\label{lem:uniform_standard}
There exists a constant $r$, with the same implicit dependence as the constants in the $O(1)$ terms, such that for all $x$, for all $\delta$ small enough, there exists standard paths connecting any point of $B_1( \cR(x, r\delta) )$ to $B_2( \cR( x, r\delta ) )$ without exiting $\cR( x, r\delta)$.
\end{corollary}

The construction of $U(\delta)$ then follows the same spirit as in \cite{BLRannex}. By \cref{lem:uniform_standard}, there exists a standard path $\cL_1$ that starts within $O(\delta)$ of $\psi(U)$, follows $\partial \phi( U)$ up to an $O(\delta)$ error, crosses itself and then goes to within $O(\delta)$ of the the boundary of $\Gamma_{ex}$. We complete this path so that it reaches the boundary of $\Gamma_{ex}$.

Erasing small loops and an initial portion of the path, we can obtain a path $\cL_2$ that starts at some point $x_0$, then does a simple loop that follows $\partial \phi( U)$ up to $O(\delta)$ error, then forms a simple path going to $\partial \Gamma_{ex}$ and avoiding its initial loop. Call $x_1$ the next vertex visited by $\cL_2$ after $x_0$ and let $\cL_3$ be the path simple starting at $x_1$ and then following $\cL_2$.

We let $\Gamma(\delta)$ be the subgraph of $\Gamma_{ex}(\delta)$ strictly inside the loop of $\cL_2$, together with boundary vertices for all vertices of the loop. It is clear that $\Gamma(\delta)$ is a T-graph and that the law of the uniform spanning tree on $\Gamma(\delta)$ is identical to the conditional law of the uniform spanning tree on $\Gamma_{ex}(\delta)$ conditioned on one of the branch being $\cL_{3}$.

We let $b_0$ be the black vertex associated to the segment containing the edge $(x_0x_1)$ and $w_0$ be the white vertex associated to the face adjacent to $(x_0, x_1)$ inside the loop of $\cL_2$. 

\begin{definition}\label{def:U}
Let $W(\delta)$ be the set of white vertex associated to a face inside the loop of $\cL_2$, except for $w_0$. Let $B(\delta) $ be the set of black vertices associated to segments that have a non-zero length inside the loop of $\cL_2$. We define $\tilde{U}(\delta)$ as the graph with vertex set $W(\delta) \cup B(\delta)$ and the standard definition of adjacency weights for a T-graph. We let $U(\delta)$ be the \emph{unweighted} graph with vertex set $W(\delta) \cup B(\delta)$ and two vertices adjacent if the associated segment and face have a long common piece.
\end{definition}

\begin{lemma}\label{lem:mapping}
The mapping between dimers and uniform spanning tree with the dual tree oriented toward $w_0$ maps the wired UST of $\Gamma(\delta)$ to dimers on $\tilde{U}(\delta)$.
\end{lemma}
\begin{proof}
As mentioned above, clearly the law of the UST in $\Gamma(\delta)$ coincides with the conditional law of a UST of $\Gamma_{ex}(\delta)$ given that one of the branch is $\cL_3$. It is also clear that under this conditional law, the dual tree (on faces of $\Gamma_{ex}$ will have a single edge connecting the inside of the loop of $\cL_2$ and the outside and that this edge will go from the face associated to $w_0$ outside going through the segment associated to $b_0$. More generally we see that the restriction of the dual tree on $\Gamma_{ex}(\delta)$ to the dual of $\Gamma(\delta)$ is determined only by the restriction of the primal tree to $\Gamma(\delta)$ and that is is always oriented towards $w_0$.
 Therefore, in the mapping between and dimers and UST on $\Gamma_{ex}$ (with any orientation of the dual) vertices in $W(\delta)$ must always be matched within $B(\delta)$.
 
 Conversely, since a simple path in a T-graph always moves from the inside of a segment to an endpoint, a single segment cannot have a non-zero length intersection with both the inside and the outside of the loop of $\cL_2$. Therefore vertices in $B(\delta)$ must be matched within $W(\delta)$. Overall we have indeed that UST of $\Gamma(\delta)$ correspond to dimer cover of $\tilde U(\delta)$.
\end{proof}

\begin{remark}
\cref{lem:mapping} might seem trivial or the proof circular considering the definition of the mapping from \cite{Kenyon2007}. The point of the lemma is to avoid the special treatment of boundary white vertices in \cref{def:dimers}, which we in fact do \emph{not} know ow to understand properly in terms of height function.
\end{remark}

\begin{proposition}\label{prop:U_boundary_height}
$U(\delta)$ is a subgraph of the hexagonal lattice which admits dimer covers and its height function satisfies $h^\d = \frac{1}{\delta} h^\cC + O(1)$ along the boundary. 
\end{proposition}
\begin{proof}
The fact that $U(\delta)$ is a subset of the hexagonal lattice follows directly from \cref{prop:geometrie}.	
	
Then we prove that $U(\delta)$ admits a dimer cover. For this, consider from each vertex a standard path in $\Gamma(\delta)$ going to the boundary, which exists by \cref{lem:uniform_standard}. Loop-erasing theses paths in any order as in Wilson's algorithm we can obtain a spanning tree of $\Gamma(\delta)$ and therefore a dimer configuration in $\tilde U(\delta)$. 

We now claim that this spanning tree uses at most one long piece of any segment. There are two ways for the UST to cover both long pieces of a segment. 
First, it is possible that at some point of the construction one path comes back to a segment $S$ already occupied by either previously drawn tree branches or its own loop-erasure. Then if the path covers the other long segment of $S$ without closing a loop then both long pieces can be occupied. Clearly this can only happen if the two long pieces of $S$ are separated by a short piece and at some point a path arrives on one side of this short piece and then does a long jump without doing a short jump, which is excluded for standard paths.
The other possibility is for both long segments to be covered ``in a row'' coming from a vertex close to the endpoint of a segment. Again this means that after arriving at a short distance of the endpoint of a segment, one of the paths did not move to that endpoint which is excluded.
This concludes because the pieces of segments left uncovered by the walk give directly the edges covered by a dimer.

To study the boundary height function, let us compute the height function along $\cL_3$. The height with respect to the flow $M_{ref}$ defined from $G_{ex}^\d$ is given by the winding of the boundary of $\Gamma(\delta)$. By construction the boundary approximates up to $O(\delta)$ a smooth curve and is a self-avoiding curve, therefore this winding stays bounded independently of $\delta$. By \cref{lem:height_macro}, this directly implies that along the boundary of $U(\delta)$, on has $h^\d = \frac{1}{\delta} h^\cC + O(1)$.
\end{proof}

\section{Extending the CLT}\label{sec:CLT}

Recall from \cref{sec:sketch} that our strategy is to apply to $\Gamma(\delta)$ the main theorem of \cite{BLR16} saying that the winding of the uniform spanning tree on a planar graph converges to the Gaussian free field. The assumptions in that theorem are mainly the uniform crossing estimate proved already in \cref{lem:uniform_crossing} and the CLT for the random walk. The purpose of this section is to prove this CLT.

The proof goes in two steps. First we extend the CLT for the random walk on the whole plane T-graphs $T_{\Delta, \lambda}$ (\cite{Laslier2013b}) so that it works for all $\lambda$ instead of almost all $\lambda$. Then we prove the result using the local convergence of the graphs $\Gamma(\delta)$ to whole plane graphs $T_{\Delta, \lambda}$. Let us note that in this second part, we essentially prove that for a martingale walk, a CLT result is always stable under local limits.
We will however not attempt to formulate a general statement to simplify the exposition. We think that it would be an interesting question to identify the  optimal assumptions for the above statement to hold.

For both steps, the key idea is that since the walk is by definition a martingale with only microscopic jumps, we only need to show that the variances of all projections grow at the same rate. Crucially, it is therefore enough to only consider large microscopic regions where we can use the local behaviour of the graph and continuity of the law of the random walk with respect to the underlying graph.

\subsection{Characterisation of convergence}

We are interested in the following uniform distance on curves up to reparametrisation.

\begin{definition}\label{def:topology}
Let $\gamma : [0, t] \to \C$, $\gamma' : [0, t'] \to \C$ be two parametrised curves in $\C$. The distance between them is
\[
d( \gamma, \gamma' ) = \inf_{\phi, \phi'} \sup_{s\in [0, 1]} | \gamma \circ \phi (s) - \gamma' \circ \phi' (s)|,
\]
where $\phi$ and $\phi'$ are continuous increasing functions from $[0, 1]$ to $[0, t]$ and $[0, t']$.
\end{definition}

It is well known that this is indeed a distance between curves seen up to reparametrisation.
We will use the following characterisation of convergence in this topology.

\begin{lemma}\label{lem:char_convergence}
	Let $X^\d_t$ be continuous time martingales in $\C$ with quadratic variation going to infinity a.s. and increments bounded by $\delta$. Assume that for all $\epsilon$ there exists a sequence of stopping times $\tau_n^\epsilon$ such that
	\begin{itemize}
		\item $\tau_{n+1}^\epsilon - \tau_{n}$ is always greater than $1$ and has an exponential tail given $\cF_{\tau_n^\epsilon}$ depending on $\epsilon$ but not $n$;
		\item $\frac{1}{\delta}(X_{t} - X_{\tau_n})$ is bounded on $[\tau_n, \tau_{n+1}]$, again uniformly on $n$ but not necessarily on $\epsilon$.
		\item For all directions $d, d' \in \cS^1$, 
		\[
		1- \epsilon \leq \frac{\Var( X_{\tau_{n+1}}\cdot d | \cF_{\tau_n})}{\Var( X_{\tau_{n+1}}\cdot d' | \cF_{\tau_n})} \leq 1+ \epsilon.
		\]
	\end{itemize}
then as $\delta \to 0$, $X^\d$ converges to a Brownian motion in the topology of \cref{def:topology}.

If $X^\d_t$ is only defined until its first exit of a open domain $U_{ex}$, then for all open domains $U $ such that $\bar U \subset U_{ex}$, the walk stopped when it exits $U$ converges to Brownian motion stopped when it exists $U$. 
\end{lemma}
\begin{proof}
	If we set $\epsilon(\delta)$ going to $0$ slowly enough as $\delta \to 0$, we can clearly define a sequence of martingales $Y^\d_n :=X^{\d}_{\tau_n^{\epsilon(\delta)}}$ such that $Y^\d$ has increments at most $\sqrt{\delta}$, quadratic variation growing to infinity and for all $d$ and such that $\Var( Y^\d_{n+1}\cdot d | \cF_{\tau_n})/\Var( Y^\d_{n+1} | \cF_{\tau_n}) \to 1/2$ uniformly.
	Clearly the martingale CLT applies to all projections $Y^\d \cdot d$ and when parametrised by the quadratic variation of $Y^\d$, all projections converge to Brownian motion with variance $1/2$.
	
	Since all projections are Gaussian, we see that any sub-sequential limit for $Y$ (parametrised still by its quadratic variation) must be a Gaussian vector. Further since the horizontal and vertical directions have the same variance and are uncorrelated, $Y^\d_n$ must converge to a standard Brownian motion. This also concludes for $X$ because by construction, $X$ and $Y$ are at most $\sqrt{\delta}$ away seen as curves up to time parametrisation. This concludes the full plane case.
	
	The second statement in a domain clearly follows from the first one if we extend $X$ by a standard Brownian motion when it exits $U_{ex}$.
\end{proof}

\subsection{CLT for all $\lambda$}\label{sec:cltTgraph}

The first obstacle in order to extend the CLT to the whole family $T_{\Delta, \lambda}$ is that the random walk is only defined on non-degenerate T-graphs while it is possible in a graph $T_{\Delta, \lambda}$ to for a face to be mapped to a single point. We therefore need to extend the definition of the random walk (however we still assume always that $\Delta$ is always a non flat triangle).

\begin{definition}
For any triangle $\Delta$ and any $\lambda$ we define the continuous random walk on $T_{\Delta, \lambda}$ as the Markov chain $X_t = X_t$ with the following transition.
\begin{itemize}
\item If $X_t$ is in a non-degenerate vertex it moves according to the general rule for T-graphs.
\item If $X_t$ is in a degenerate point, let $w_0$ be the associated white vertex and $b_1, b_2, b_3$ the black vertices adjacent to $w_0$ and let $L_1, L_2, L_3$ be the length of the associated segments. Let $\ell_1, \ell_2, \ell_3$ be the edge lengths in the triangle $\Delta$ with $\ell_1$ associated to edges of the type $b_1 w_0$ and similarly for the others. $X$ jumps to the other side of the segment associated to $b_1$ with rate $\frac{\ell_1}{L_1( \ell_1 + \ell_2 + \ell3)}$.
\end{itemize}
\end{definition}

\begin{remark}
$\mu_i =\ell_i/(\ell_1+\ell_2+\ell_3)$ is the invariant measure for a chain jumping cyclically  along $\Delta$ with rate $1/\ell_i$.
\end{remark}

\begin{notation}
For the rest of this section, our argument will involve comparisons between the laws of random walks on different graphs. Starting from now, we will therefore write $X(\Delta, \lambda)$ for the random walk started at the point of coordinate $(0, 0)$ in $T_{\Delta, \lambda}$ and $X(\Gamma(\delta), v)$ for the random walk on the graph $\Gamma(\delta)$ started from $v$ (of course for $v \in \Gamma(\delta)$). To lighten notations, we will however still drop these parameters when they are clear from the context. For all these comparison purposes, we see the walks as processes in $\C$ and we will always consider the uniform distance on curves up to re-parametrisation when speaking of the continuity of these walks.
\end{notation}

\begin{proposition}\label{lem:continuity_walk}

For all $t_{\max}$, the law of $(X_t(\Delta, \lambda))_{0\leq t \leq t_{\max}}$ is continuous as a function of $(\Delta, \lambda)$.
\end{proposition}
\begin{proof}
Fix $\epsilon > 0$. It is easy to see that even starting at a degenerate vertex, the variance does not grow faster than linearly with time so the upper bound in \cref{lem:variance} still holds. We can therefore find $R$ such that $\bbP(\max_{t \leq T}|X_t| \geq R ) \leq \epsilon$ and consequently  we do not need to look at the graph beyond radius $R$ in the T-graph. By almost linearity this part of the T-graph only depends on the flow up to some radius $R'$ in $\cH$ and therefore the image of any point is a continuous function of $\lambda$ and $\Delta$.

Consider a pair $(\Delta, \lambda)$ such that no triangle is degenerate in $B(0, R)$. This stays true in a neighbourhood of $(\Delta, \lambda)$ so the graph structure of the T-graph is constant in this neighbourhood. The transition rates are also continuous functions of $(\Delta, \lambda)$ because they only depend on the embedding so overall $(\Delta, \lambda)$ is a continuity point of the law of $X(\Delta, \lambda)$.

Now consider the case where there is a degenerate triangle for $((\Delta_0, \lambda_0))$. Let us call the corresponding white vertex $w_0$ and let $b_1, b_2, b_3$ be the three black vertices adjacent to $w_0$. Note that for any $(\Delta, \lambda)$, the three edges adjacent to $\psi(w_0)$ are the ones corresponding to $b_1, b_2, b_3$. Furthermore the length $L_1(\Delta, \lambda)$, $L_2(\Delta, \lambda)$ and $L_3(\Delta, \lambda)$ of these edges are continuous functions of $(\Delta, \lambda)$ so the jump rate $1/L_i$ for the ``long'' jump along these edges are continuous functions of $(\Delta, \lambda)$

Now for any $(\Delta, \lambda)$ close enough to $(\Delta_0, \lambda_0)$, when the walk starts on one of the vertices of $\psi(w_0)$, it will jump with very high rate around $w_0$ before moving out. 
This means that it will mix in a short time and that the effective transition rate for the three ``long'' jumps along the $\psi(b_i)$ are just the original weights multiplied by the invariant measure for jumping around $\psi(w_0)$. Now note that if $\psi(w_0) = r \Delta$ up to rotation, then the rate of transition along the small portion of $\psi(b_1)$ is $r/\ell_1$, independently of whether the corresponding vertex is at the intersection of $\psi(b_1)$ and $\psi(b_2)$ or of $\psi(b_1)$ and $\psi(b_3)$. Therefore the effective rate of jump along the long part of $\psi(b_1)$ is close to $\frac{1}{L_1} \frac{\ell_1}{\ell_1 + \ell_2+\ell_3}$ independently of the local geometry, which proves that $(\Delta_0, \lambda_0)$ is a continuity point.
\end{proof}

\begin{remark}
The proof of the proposition explains why we cannot consistently define the random walk on a general T-graph started from a degenerate point. In fact the walk behaves as if a degenerate face still has a non-trivial geometry (in our case always a triangle) which matters in the definition of the random walk. If we are just given a T-graph with degenerate face, there is no canonical way to fix this geometry. In our case, we are only saved by the fact that all faces are scaled versions of the same triangle so we know we have to use this geometry for the degenerate face.
\end{remark}

\begin{proposition}
For all $\Delta$, $\lambda$, the random walk on $\delta T_{\Delta, \lambda}$ converges as $\delta$ goes to $0$ to a constant multiple of the standard Brownian motion. The convergence holds in law and for the uniform topology on curves up to time re-parametrisation.
\end{proposition}

\begin{proof}
Fix $\Delta$, we known that the results holds for almost all $\lambda$ with a certain multiple $\sigma$ of the Brownian motion which is independent of $\lambda$, furthermore using the exponential concentration bound in \cref{prop:concentration}, it is easy to see that all moments also converge. When $\delta$ has only rational angles, the T-graph is periodic and the convergence is trivial, therefore we assume from now on the $\Delta$ has at least one irrational angle.
Rewriting the definition of convergence, \cref{thm:CLT_ref} gives in particular:
\[
\text{for almost every } \lambda,  \forall \epsilon,\exists n_0,\, \forall n \geq n_0 \, \forall d \in \mathbb{S}_1,\, \sigma^2 n (1-\epsilon) \leq \E( X_n(\lambda).d )^2 \leq \sigma^2 n (1+\epsilon).
\]
Therefore for a fixed $\epsilon > 0$ and $n_0$ large enough, the measure of the set of parameters $\lambda$ satisfying the above property must be large. Combined with \cref{lem:continuity_walk} this implies that we can find an open set $\Lambda$ and $n >0$ such that
\begin{itemize}
\item The Lebesgue measure of $\Lambda$ is greater than $2\pi-\epsilon$,
\item $\forall d, \forall \lambda \in \Lambda, \sigma^2 n (1-\epsilon) \leq \E( X_{n}(\lambda).d )^2 \leq \sigma^2 n (1+\epsilon)$,
\end{itemize}
where we emphasize that now the quantification on $\lambda$ is not almost sure any more. In this proof we will say that an event happen happens with small (resp. hight) probability if its probability goes to $0$ (resp. $1$) as $\epsilon$ goes to $0$

Now we fix $\lambda_0$, and for all vertices $v$ of $T_{\Delta, \lambda_0}$, we define $\lambda (v)$ as the effective $\lambda$ around $v$ in the sense that $T_{\Delta, \lambda(v)}+ v = T_{\Delta, \lambda_0}$. (If $\pi p_a, \pi p_b$ and $\pi p_c$ are the angles of $\Delta$ and $v$ has coordinates $(m,n)$ we have $\lambda (v) = \lambda(0) e^{i \pi ( (1-p_c)m + (p_c+p_b)n )}$ but the exact formula is not important). We will call a vertex $v$ ``good'' if $\lambda(v) \in \Lambda$ and bad otherwise. As $R$ goes no infinity, since $\Delta$ has at least one irrational angle, the uniform measure on $\{\lambda (v) \mid v \in B(0, R) \}$ converges to the uniform measure on the circle . In particular, for all $R$ large enough the number of bad points in $B(0, R)$ is at most $C\epsilon R^2$ for some constant $C$.

Consider the time-changed walk $Y_k = X_{t_k}$ where the times $t_k$ are defined by induction $t_{k+1} = t_k + n$ if $Y_k$ is a good vertex and $t_{k+1} = t_k +1$ if it is a bad point. Clearly by the martingale CLT, for all directions $d$, $Y.d$ converges to Brownian motion up to a possibly random time change. Fix $R$ large enough and define inductively stopping times $\kappa_i$ by $\kappa_{i+1} = \inf \{ k \geq \kappa_i : Y_k \notin B( Y_{\kappa_i}, R ) \}$.  Calling the number of visits to good (resp. bad) points between $\kappa_i$ and $\kappa_{i+1}$ $G_i$ (resp. $B_i$), by definition of good and bad points we have
\[
 \sigma^2 (1-\epsilon) n G_i \leq \langle Y_{\kappa_{i+1}.d} \rangle - \langle Y_{\kappa_{i}.d} \rangle \leq  \sigma^2 (1+\epsilon) n G_i + C B_i.
\]
On the other hand, by \cref{lem:nb_visite} applied to the indicator of bad points, $\E( B_i ) \leq C \epsilon R^2$ for all $i$. Finally, by \cref{prop:concentration} the probability that $\kappa_{i+1} - \kappa_i \leq c R^2/n$ goes to $0$ as $c$ goes to $0$. In particular taking $c= \sqrt{\epsilon}$, we see that the probability that $\langle Y_{\kappa_{i+1}} \rangle - \langle Y_{\kappa_{i}} \rangle \leq \sqrt{\epsilon} R^2$ is small. In particular with hight probability we have
\[
\sigma^2 (1 + \epsilon) n G_i \geq (\sqrt{\epsilon}- C \epsilon ) R^2
\]
for some constant $C$. Putting everything together, for any $\epsilon$ small enough we have with high probability
\[
\sigma^2 (1-\epsilon') n G_i \leq \langle Y_{\kappa_{i+1}}.d \rangle - \langle Y_{\kappa_{i}}.d \rangle \leq  \sigma^2 (1+\epsilon') n G_i.
\]
Since the definition of $G_i$ is independent of $d$, this shows that the variance of $(Y_{\kappa_i})_{i \geq 0}$ is almost identical in all directions. Since the sequence $Y_{\kappa_{i}}$ is just a subsequence of $X$ along stopping times, we can apply \cref{lem:char_convergence} and obtain the CLT as desired.
\end{proof}

\subsection{CLT under a local limit assumption}

In this section we prove that the random walk on $\Gamma(\delta)$ converges to a Brownian motion on $\phi(U)$ until it touches the boundary.

\begin{lemma}\label{lem:continity_RW}
For all $\epsilon > 0$ and for all $t_{\max} >0$, there exists $\delta_0$ such that for all $\delta < \delta_0$, for all $v \in \Gamma(\delta)$ the following holds.
\[
d\Big( \frac{1}{\delta} \big(X_t( \Gamma(\delta), v)-v\big)_{0 \leq t \leq t_{\max}}, \big(X_t( \Delta(v), \lambda(v)\big)_{0 \leq t \leq t_{\max}} \Big) \leq \epsilon,
\]
where $\Delta(v)$ and $\lambda(v)$ are defined by the formula in \cref{lem:meso_psi} and with a slight abuse of notation $d$ is the Lévy–Prokhorov metric on the laws of these walks.
\end{lemma}
\begin{proof}
Since the transition rates of the walks are given in terms of the geometry of the graphs and by \cref{lem:meso_psi} $\frac{1}{\delta}\Gamma(\delta)$ and $T_{\Delta(v), \lambda(v)}$ are very close, there is almost nothing to prove. The only issue is that near a degenerate face, the transition rates are not a continuous function of the graph for the Hausdorff distance (indeed they involve the aspect ratios in a small face). However by construction (see \cref{lem:choice_lambda,lem:isolated_vertices}), each face of $\Gamma$ has almost the same aspect ratio as in $\psi$ since edges had initial length at least $\Theta(\delta^3)$ and were modified by at most $O(\delta^{13})$. By \cref{lem:psiw} this original aspect ratio is the same as in $T$, therefore the law of the random walk are also close near degenerate faces which proves the lemma.
\end{proof}

\begin{proposition}\label{prop:CLT_Gamma}
For every $x$ in $\phi( U)$, for every sequence $x^\d$ of vertices of $\Gamma(\delta)$ converging to $x$, the random walk $X^\d$ started in $x^\d$ (and killed when it touches the boundary) converges to Brownian Motion in $\phi (U)$ (also killed when it touches the boundary). The convergence is in law and for the uniform topology on curves up to re-parametrization.
\end{proposition}

\begin{proof}
Fix $\epsilon > 0$.
 For any $R$, we consider the of stopping times $\tau_{R} = \inf \{ t : X_t \notin B( X_{0}, R ) \}$. Since the random walk converges to Brownian motion, for every $\lambda, \Delta$, we can find $R_0$ such that for all $R \geq R_0$
\[
\forall d,\, \forall R \geq R_0 \, \sigma^2 R^2 (1-\epsilon) \leq \E( X_{\tau_R}((\Delta, \lambda)) . d)^2 \leq \sigma^2 R^2 (1+\epsilon).
\]
By continuity of the law of $X(\Delta, \lambda)$ with respect to $\Delta, \lambda$, we can find a ball around each point where the same equation holds if we allow for an error $2\epsilon$. 

Now in $\Gamma(\delta)$, uniformly over $v$, one can approximate $B(v, \delta R)\cap \Gamma(\delta)$ by of $\delta B(v, R)\cap T( \Delta(v), \lambda(v))$ by \cref{lem:meso_psi}. Furthermore, since the parameter $\Delta(v)$ is given explicitly in terms of continuous functions extending in a neighbourhood of $U$, we can find an ``apriori''  compact set of triangles $K$ such that $\Delta(v) \in K$ for all $\delta$ and for all $v \in \Gamma(\delta)$. By extracting a finite cover of $K \times \mathbb{S}^1$, we obtain the following
\[
\forall v \in \Gamma(\delta), \exists R (v) \, \forall d,\, \sigma^2 R(v)^2 (1-2\epsilon) \leq \E_{T_{\Delta(v), \lambda(v)}}( X_{\tau_{R(v)}} . d)^2 \leq \sigma^2 R(v)^2 (1+2\epsilon)
\]
where $R(v)$ takes only finitely many values.

Over a fixed domain, we know that $\Gamma(\delta)$ is close to $\delta T( \Delta(v), \lambda(v))$ (uniformly in $v$) and the law of the walk depends continuously on the graph by \cref{lem:continity_RW}. Therefore the walk in $\Gamma (\delta)$ is uniformly well approximated by the walk on $\delta T_{\Delta(v), \lambda(v)}$ up to $\tau_{R(v)}$ and for all $\delta$ small enough, we obtain
\[
\forall d, \forall v \in \Gamma(\delta), \, \sigma^2 R(v)^2 (1-3 \epsilon) \delta^2 \leq \E_v( X_{\tau_{\delta R(v)}}(\Gamma(\delta)) . d)^2 \leq \sigma^2 R(v)^2 (1+3 \epsilon) \delta^2.
\]
We conclude by \cref{lem:char_convergence} as above.
\end{proof}

\section{Comparisons of measures, conclusion}\label{sec:comparison}

At this point, we can apply the result of \cite{BLR16} to see that the fluctuations of the UST winding field on $\Gamma(\delta)$ converge to the Dirichlet Gaussian free field, which means (using the ocnvergence of the map from $\tilde U(\delta)$ to $\Gamma(\delta)$ from \cref{lem:global_psi}) that the height function on $\tilde U(\delta)$ converges to the desired transformation of a GFF. We are however not done yet because (as mentioned in \cref{sec:graph}), $\tilde U(\delta)$ is not an hexagonal lattice with uniform weights : as per \cref{prop:geometrie}, $U(\delta)$ has extra edges compared to the hexagonal lattice and the weights on the original edges are not completely uniform. The purpose of this section is to control the effect of these changes to conclude the proof of \cref{prop:main}.
We will proceed in two steps : First we show using Wilson's algorithm that with high probability the dimer measure on $\tilde U(\delta)$ does not use any of the extra edges in $\tilde U(\delta) \setminus U(\delta)$. Then we compare the laws by a direct comparison of the weight of each configuration.

 Recall from \cref{prop:segments} that any segment contains two long pieces which correspond to edges of $\cH$ and up to three short pieces of length $O(\delta^{13})$, one possibly between the long pieces and one adjacent to each endpoint. Also recall from \cref{prop:geometrie} that each vertex of $\Gamma(\delta)$ is adjacent to at most one short piece.

\begin{lemma}\label{lem:occupation_tree}
In a spanning tree $T$ of $\Gamma$, the probability that there is a segment $S$ on which both ``long'' pieces are occupied in $T$ is a most $O(\delta^2)$. On this event the dimer model on $U(\delta)$ only uses edges of $U_{\cH}(\delta)$.
\end{lemma}
\begin{proof}
	It is essentially a direct application of \cref{lem:standard} and \cref{prop:geometrie}. More precisely, since one needs to run $O(\delta^{-2})$ during Wilson's algorithm, the probability that all walks are standard is at least $1 - O(\delta^{2})$.

We now claim that it is a deterministic fact that if we loop-erase a set of standard random walks, the resulting spanning tree will use at most one long piece per segment. Indeed there are two ways for the UST to cover both long pieces of a segment $S$. First it is possible that at some point one long piece is covered by the current loop-erasure of the random walk or some previous branch while the walk comes back to the segment and covers the other long segment without closing a loop. Clearly this can only happen if the two long pieces of $S$ are separated by a short piece and at some point the walk arrives on one side of this short piece and then does a long jump without doing a short jump (and therefore is not standard). The other possibility is for both long segments to be covered ``in a row'' coming from a vertex close to the endpoint of a segment. Again this means that after arriving at a short distance of the endpoint of a segment, the walk did not move to that endpoint.

This concludes because the pieces of segments left uncovered by the walk give directly the edges covered by a dimer. 
\end{proof}

\begin{lemma}\label{prop:comparison}
Let $U_{\cH}(\delta)$ be the graph defined by the set of long edges of $\tilde U$ and with weights inherited from $\tilde U$. The dimer measures on $U_{\cH}(\delta)$ and $U(\delta)$ are absolutely continuous with respect to one other and their Radon-Nidodym derivative is equal to $1+O(\delta^2)$ uniformly over all configurations.
\end{lemma}

\begin{proof}
Note that the graphs $U_\cH(\delta)$ and $U(\delta)$ have the same vertices and edges, only their edge weights differ. It is therefore clear that their dimer laws are absolutely continuous with respect to each other and that the Radon-Nicodym derivative $R$ is just the ratio of all edge weights divided by a partition function.

Consider two dimer configurations $M$ and $M'$ on $U(\delta)$ that differ only by a rotation around a face $f$, It is clear that
\[
\frac{R(M')}{R(M)} = \left| \prod_i \frac{\tilde w(b_i, w_i)}{\tilde w(b_i, w_{i+1})} \right|
\]
where the $w_i$ and $b_i$ are the black and white vertices around $f$ and $\tilde{w}$ denote the weights of edges in $\tilde U(\delta)$. By \cref{prop:comp_K} this is equal to $1+ 0(\delta^5)$. On the other hand, it is well known that all dimer configurations on simply connected subgraph of the hexagonal lattice can be connected with at most $O(\delta^3)$ rotations around faces. Therefore $R(M) = 1+0(\delta^2)$ uniformly over all configurations as desired.
\end{proof}

To prove the main proposition from \cref{sec:setup}, we just need to put our pieces together.
\begin{proof}[Proof of \cref{prop:main}]
Obviously we will take the domain $U(\delta)$ in the statement of \cref{prop:main} to be the domain $U(\delta)$ from \cref{def:U}. \Cref{prop:main} contains three statement that we need to check :
\begin{enumerate}
\item Seen as a closed set of $\bbC$, $U(\delta)$ is at Hausdorff distance $O(\delta)$ of $U$,
\item $\sup_{v \in \partial U(\delta)} |\delta h^{\d} (v) - h(v)| = 0(\delta)$,
\item There exists a map $\phi$ from $U_{ex}$ to $\bbD$ such that
\[
h^\d - \E( h^\d ) \to \frac{1}{2 \pi \chi} h^{GFF} \circ \phi.
\] 
\end{enumerate}
Items (1) and (2) were already proved in \cref{prop:U_boundary_height}. More precisely \cref{prop:U_boundary_height} states that $U(\delta) = U^\d$ away from the boundary and in particular if we embed both using regular hexagons $d_\cH (U(\delta, U^\d) ) = O(\delta)$. On the other hand by construction $d_\cH (U^\d, U) = O(\delta)$ which proves the first point. The second point is explicitly stated in the proposition.

For item (3), we again obviously want to use the function $\phi$ defined in \cref{prop:def:phi}. By \cref{prop:CLT_Gamma}, the random walk on the graphs $\Gamma (\delta)$ satisfy a central limit theorem. They also clearly satisfy the other assumptions for the main theorem of \cite{BLR16} (the main so called ``uniform crossing'' estimate is stated in \cref{lem:uniform_crossing}) so if we write $h_\Gamma^\d$ its winding field, we have
\[
h_\Gamma^\d - \E( h_\Gamma^\d ) \to \frac{1}{\chi}  h^{GFF}
\]
where $h^{GFF}$ is a Dirichlet Gaussian free field in $\phi(U)$. The convergence is in a strong moment sense, see \cite{BLR16} for details. Translating this convergence in terms of dimer height function, it exactly means that for the height function $h_{\tilde U}$
\[
h_{\tilde U}^\d- \E(h_{\tilde U}^\d) \to \frac{1}{2 \pi \chi} h^{GFF} \circ \phi.
\]
since the map from $\tilde U$ to $\Gamma$ is asymptotically given by $\phi$.
Now by \cref{lem:occupation_tree}, the dimers measures on $\tilde U(\delta)$ and on $U_{\cH} (\delta)$ are at most $O(\delta^{-2})$ away in total variation. By \cref{prop:comparison}, this measure is also at most $O(\delta^2)$ away in total variation from the measure on $U(\delta)$. Overall we get
\[
h^\d - \E( h_\Gamma^\d ) \to \frac{1}{2 \pi \chi}  h^{GFF} \circ \phi.
\]

We still need to replace $\E( h_\Gamma^\d )$ by $\E( h^\d )$ and for this we will bound moments. By applying Lemma 5.5 in \cite{BLR16} to two neighbouring points, together with Proposition 5.4 there to control the $m^\d ( v_i) - \E( h^\d(v_i) )$ term, we have for all $\delta$ small enough, for all $v$
\[
\Var( h_{\tilde U}^\d(v) ) \leq C (1 + \log^{4} \delta )
\]
for some $C$. Since $h_{\cH}$ is just a conditioned version of $h_{\tilde U}^\d$, we have
\[
\Var( h^\d_{\cH}(v) ) \leq \E[( h^\d_{\cH}(v) - \E h^\d_{\tilde U}(v) )^2] \leq (1+ O(\delta^2) )C (1 + \log^{4} \delta ).
\]
Similarly, since the Radon-Nycodym derivative of $h_{\cH}^\d$ with respect to $h^\d$ is of order $1+0(\delta^2)$, we also have
\[
\Var( h^\d(v) ) \leq (1+ O(\delta^2) )C (1 + \log^{4} \delta ).
\]
Together with the fact that the total variation distance between $h_{\tilde U}^\d$ and $h^\d$ is at most $O(\delta^2)$, this shows that $\E( h_{\tilde U}^\d ) - \E( h^\d ) = O(\delta^2 \log^2 \delta )$ which concludes the proof.
\end{proof}

\begin{proof}[Proof of \cref{prop:moment}]
This was essentially already done in the proof of \cref{prop:main} just above. Note that Lemma 5.5 in \cite{BLR16} also applies to higher moment than 2 so we have for all $k$ and $v$
\[
\E[  |h_{\tilde U} (v) - \E[ h_{\tilde U}(v)]|^k ] \leq C_k (1 + \log^{2k} (\delta).
\]
We already proved that $\E[ h_{\tilde U}(v)] - \E[ h^\d(v) ] = O(\delta^2 \log^2\delta )$ so we can conclude exactly as in the proof of \cref{prop:main}.
\end{proof}

\begin{proof}[Proof of \cref{prop:esperance}]
Note that for this statement, since there is no centring of the height, the choice of reference flow is important. On the other hand, apart from that, the comparison between $\E h^\d_{\tilde U}$ and $\E h^\d$ is already done in the proof of \cref{prop:main}. Let us therefore use the standard reference flow of $1/3$ along every edge of $\cH$ to compute $h^\d_{\tilde U}$ so that it is enough to prove
\[
|\delta \E h^\d_{\tilde U} - h^\cC | \leq C \delta ( 1 + \log \frac{1}{\delta}).
\]
Let $h_\Gamma^\d$ be the height computed using the flow $M_{ref}$ so that $h_\Gamma^\d$ matches directly the winding of the spanning tree. 
By Proposition 4.12 in \cite{BLR16}, we see that each exponential scale contribute at most $O(1)$ to the expected winding of a branch. On the other hand, the minimal size of a face on the graph $\Gamma(\delta)$ is of order $\delta^3$ so
\[
|\E h^\d_{\Gamma} | \leq C( 1 + \log \frac{1}{\delta} ).
\]
On the other hand, the only difference between $h^\d_{\gamma}$ and $ h^\d_{\tilde U}$ comes from the change of reference flow and a global scaling so
\[
h_{\tilde{U}}^\d =  h^\d_{\Gamma} + h_{ref}.
\]
The control of $h_{ref}$ from \cref{lem:height_macro} concludes the proof.
\end{proof}

\appendix

\section{Continuity with respect to $h^\cC$}\label{app:continuity}

In this section, we prove the continuity statements in \cref{prop:main,prop:moment,prop:esperance}. The main difficult points are the continuity of $\phi$, and the functions $M^{\pm, \delta}$ so the proof is analytical in nature. We will not try to be optimise the topology we use on $h^\cC$ instead trying to make the argument as easy as possible even for reader not familiar with PDEs. In particular, we will not discuss any estimate on the regularity of $h^\cC$ coming from the fact that it solves a (non-linear) elliptic equation. We only note that to the best of our admittedly limited knowledge, the fact that the PDE satisfied by limit shapes admits solutions with frozen faces makes it hard to apply directly the standard theory of elliptic operators.

We first recall standard results.
The state $W^{k, p}$ for $k \in \bbN$ and $p \in [1, \infty]$ is the space of distributions which are represented together with all their partial derivatives by functions in $L^p$. It is a Banach space for the norm
\[
\norm{f}_{W^{k, p}} = \sum_{i+j \leq k} \norm{\frac{\partial^{i+j} f}{\partial^i x \partial^j y}}_{L^p} .
\]
For $p > 2$, elements of $W^{1, p}$ are $1 - 2/p$ Hölder functions and satisfy (Theorem A.6.1 in \cite{Astala2008}).
\[
|f(x) - f(y)| \leq \frac{4p}{p-2}|x-y|^{1-p/2}\norm{\nabla f}_p
\]
so in particular for $p > 2$ and in a bounded domain, the $W^{k, p}$ norm is stronger than the uniform norm on all derivatives up to order $k-1$. We will denote as usual the space of $\alpha$-Hölder functions by $\cC^\alpha$ and by $\cC^{k, \alpha}$ the space of functions whose derivatives of order $k$ are $\alpha$-Hölder.

We will also need some results on the solution of classical PDEs.
\begin{theorem}[Theorem 4.7.2 in \cite{Astala2008}]\label{thm:cauchy}
If $h$ is a $\cC^\alpha$ function on $\C$, then the equation
\[
\frac{\partial f}{\partial \bar z} = h
\]
admits an unique solution in $\cC^{1, \alpha}$, up to a linear function $a + bz$. It satisfies the estimate
\[
\norm{\frac{\partial f}{\partial z}}_{\cC^\alpha} + \norm{\frac{\partial  f}{\partial \bar z}}_{\cC^\alpha} \leq \frac{5}{\alpha(1-\alpha)}\norm{ h }_{\cC^\alpha}.
\]
\end{theorem}

\begin{theorem}[Theorem 4.5.3 in \cite{Astala2008}]\label{thm:boundedS}
For any $p > 1$, there exists $C > 0$ such that for any $f \in W^{1, p}(\C)$,
\[
\norm{ \frac{\partial f}{\partial z} }_{L^p(\C)} \leq C \norm{ \frac{\partial f}{\partial \bar z}}_{L^p(\C)}.
\]
\end{theorem}
Note that holomorphic functions cannot be in $L^p( \C )$ so the fact that $\frac{\partial f}{\partial \bar z}$ only determines a function $f$ up to an holomorphic term does not contradict our theorem.

\begin{theorem}[Theorem 5.1.1 in \cite{Astala2008}]\label{thm:inhomogeneous_solution}
For every $\mu_{\mathrm{max}} \in (0, 1)$, there exists $ P> 2$ such that the following holds for all $p \in (2, P)$. If $\mu$ and $\phi$ are functions with compact support and if $\norm{\mu}_\infty \leq \mu_{\mathrm{max}}$ and $\phi \in L^p$ then the equation
\[
\frac{\partial f}{\partial \bar z} = \mu \frac{\partial f}{\partial z} + \phi
\]
has a unique solution with $ f \in W^{1,p}$ and $f(z) = O(1/z)$ at infinity. Furthermore the solution operator $\phi \to f$ is bounded from $L^p$ to $W^{1,p}$.
\end{theorem}

\begin{theorem}[Theorem 5.1.2, Theorem 5.3.2 and equation 5.10 in \cite{Astala2008}]\label{thm:measurable_riemann}
For every $\mu_{\mathrm{max}} \in (0, 1)$, there exists $ P> 2$ such that the following holds for all $p \in (2, P)$. For all $\mu$ with compact support and $\norm{\mu}_{\infty} \leq \mu_{\mathrm{max}}$, there exists a unique function $f \in W^{1, 2}_{\mathrm{loc}}$ such that
\begin{align*}
\frac{\partial f}{\partial \bar z}& = \mu \frac{\partial f}{\partial z} \\
f(z) & = z + O(1/z) \text{ at infinity}.
\end{align*}
This solution is called the principal solution of the equation $\frac{\partial f}{\partial \bar z} = \mu \frac{\partial f}{\partial z}$. Furthermore $f$ is an homeomorphism of $\C$ and $f(z) - z \in W^{1, p}( \C )$.
\end{theorem}

\begin{lemma}[Lemma 5.3.1 in \cite{Astala2008}]\label{lem:continuity_solution}
For every $\mu_{\mathrm{max}} \in (0, 1)$, there exists $ P> 2$ such that the following holds for all pairs $p, s$ such that $2 < p < ps < P$. For all $\mu_1, \mu_2$ satisfying $ |\mu_i| \leq \mu_0 1_{B(0, r)}$ for some $r$, the principal solutions $f_1, f_2$ for $\mu_1$ and $\mu_2$ satisfy
\[
\norm{ \frac{\partial f_1}{\partial \bar z} - \frac{\partial f_2}{\partial \bar z} }_{L^p} \leq C( p, s, \mu_{\mathrm{max}}) r^{\frac{2}{ps}} \norm{ \mu_1 - \mu_2 }_{L^{ps/(s-1)}}.
\]
\end{lemma}

 We can now turn to the continuity of the function $\phi$ with respect to $h^\cC$. We fix domains $U$ and $U_{ex}$, both inside $B(0, r)$. We let $h_1$ and $h_2$ be two non-extremal dimer limit shapes, dropping the superscript since we are always in the continuum in this section. We let $\Phi_1$ and $\Phi_2$ be as in \cref{def:Phi} and we let 
 \[
 \mu_i = \frac{\frac{d\Phi_i}{d \bar z}}{\frac{d\Phi_i}{dz}} = \frac{\Phi - e^{i\pi/3}}{\Phi-e^{-i\pi/3}}.
 \]
Since the $h_i$ are non extremal we have $|\mu_i| \leq {\mathrm{max}}$ for some $\mu_{\mathrm{max}} < 1$ and actually note that the value of $\mu_{\mathrm{max}}$ controls how close to extremal the gradients of the $h_i$ are. We let $3 \geq P>2$ be chosen as in \cref{thm:inhomogeneous_solution,thm:measurable_riemann} and to ease notations we will let $C$ denote constants which can change line by line which only depend on the parameters given explicitly as in \cref{lem:continuity_solution}.
 
\begin{lemma}
If $h_1$ and $h_2$ are in $W^{35, 3} (U_{ex})$ then one can extend to $\mu_i$ to $\C$ so that $\mu_i \in W^{34, 3}$ and the principal solution $f_i$ are in $W^{34, p}(U_{ex})$ with $\frac{\partial f_i}{\partial \bar z} \in W^{33, p}(\C)$ for all $p < P$. Furthermore we have
\begin{align*}
\norm{ \mu_i }_{W^{34, 3}(\C)} & \leq C( \mu_{\mathrm{max}}) \norm{ h_i}_{W^{34, 3}( U_{ex} )} \\
\norm{ \mu_1 - \mu_2}_{W^{34, 3}(\C)} & \leq C(\mu_{\mathrm{max}}) \norm{ h_1 - h_2}_{W^{35, 3}(U_{ex})} \\
\norm{f_1 - f_2}_{W^{34, p}(\C)} & \leq C( \mu_{\mathrm{max}}, r, p ) \norm{ h_1 - h_2}_{W^{35, 3}(U_{ex})}
\end{align*}
\end{lemma}
\begin{proof}
Note that the function $\nabla h_i \to \Phi_i$ is smooth over all non-extremal gradients. In particular $\Phi_i \in W^{34, 3}$ and we can bound its norm in term of $h_i$. Similarly we have $\mu_i = \frac{\Phi_i - e^{i\pi/3}}{\Phi - e^{-i\pi/3}}$ and since $\Im( \Phi) > 0$ we easily obtain the first two points.

For the last point, we first note that for any differential operator $D$, $Df_i$ satisfies the equation
\[
\frac{ \partial D f_i}{\partial \bar z} = \mu \frac{\partial Df_i}{\partial z} + D \mu \frac{\partial f_i}{\partial z}.
\]
Note that this equation is of the type of \cref{thm:inhomogeneous_solution} since $D_mu$ has compact support, $D\mu$ is continuous because $\mu \in W^{34, 3}$ and $\frac{\partial f_i - z}{\partial z} \in L^{p}$ by \cref{thm:measurable_riemann}. Therefore by uniqueness in \cref{thm:inhomogeneous_solution}, $D f_i - Dz \in W^{1, p}$ for all $p$ and since this was true for all differential operator, $f_i - z \in W^{2, p}$.
However note that if $D'$ is another differential operator, we have
\[
\frac{ \partial D' D f_i}{\partial \bar z} = \mu \frac{\partial D' Df_i}{\partial z} + D' \mu \frac{\partial D f_i}{\partial z} + D' D \mu \frac{\partial f_i}{\partial z} + D \mu \frac{ \partial D' f_i}{\partial z}.
\]
and we can reapply the argument above to get that $f_i - z \in W^{3, p}$. Clearly we can keep iterating until we obtain $f_i -z \in W^{34, p}$.

Now we note that $f_1-f_2$ satisfies the equation
For the last point, we note that $f_1 - f_2$ satisfies the equation
\[
\frac{\partial (f_1 - f_2)}{\partial \bar z} = \frac{\mu_1 + \mu_2}{2} \frac{\partial (f_1 - f_2)}{\partial z} + \frac{\mu_1 - \mu_2}{2} \frac{\partial (f_1 + f_2)}{\partial z}
\]
with condition $f_1 - f_2 = O(1/z)$ at infinity. This is again of the type of \cref{thm:inhomogeneous_solution}. On the other hand, by \cref{lem:continuity_solution}, 
\[
\norm{\frac{\partial (f_1-f_2)}{\partial \bar z}}_{L^p} \leq C( \mu_{\mathrm{max}}, p, r) \norm{ \mu_1 - \mu_2}_{L^{\infty}}\leq  C( \mu_{\mathrm{max}}, p, r) \norm{ \mu_1 - \mu_2}_{W^{1, 3}},
\]
By \cref{thm:boundedS}, the bound on $\frac{\partial (f_1-f_2)}{\partial \bar z}$ extends to $\frac{\partial (f_1-f_2)}{\partial z}$ and since the solution operator in \cref{thm:inhomogeneous_solution} is continuous, we obtain
\[
\norm{f_1 - f_2}_{W^{1,p}} \leq C( \mu_{\mathrm{max}}, p, r) \norm{\mu_1 - \mu_2}_{W^{1, 3}}.
\]
We can then iterate as above to bound the derivatives.
\end{proof}

In order to speak of the continuity of the choice of function $\phi$, it is necessary to have a consistent choice of conformal map between $f_i(U_{ex})$ and $\bbD$. for this we fix $z \in U$ and we take the convention that $f_i(z)$ is sent to $0$ with real positive derivative. We call $g_i$ the conformal map from $\bbD$ to $f_i( U_{ex} )$ so that $\phi = g_i^{-1} \circ f_i$.

\begin{lemma}
With the above convention, the function $\phi$ depends continuously on the function $h^\cC$ when we consider the $W^{34, p}_{\mathrm{loc}}(U_{ex})$ topology for $\phi$ and the $W^{35, 3}( U_{ex} )$ topology on $h^\cC$.
\end{lemma}
\begin{proof}
First we note that $f(U_{ex})$ is a Jordan domain whose boundary depends continuously on $h$ (say for the uniform topology). It is easy to see that this implies by the Caratheodory kernel theorem that $g$ depends continuously on $h$ for the uniform topology on compacts. Since the $g$ are conformal maps, the convergence of the functions also implies the convergence of all their derivatives.

Since $g$ is a conformal map, it has a non-zero derivative inside $\bbD$ and therefore its inverse must also depend continuously on $h^\cC$. We have already proved that $f$ depends continuously on $h^\cC$ and the composition of a function in $W^{34, p}$ with a $\cC^\infty$ function cannot change the local regularity, therefore $\phi$ depends continuously on $h^\cC$ as desired.
\end{proof}

\begin{lemma}\label{lem:continuityMn}
The functions $M_n^\pm$ defined in the proof of \cref{lem:def_FG} for $1 \leq n \leq 30$ depend continuously on $h^\cC$ when we consider the $\cC^{34-n}$ topology on $M_n^\pm$ and the $W^{35, 3}$ topology on $h^\cC$. In particular for $\delta \leq 1$, the functions $M^{+ \delta}$ and $M^{-, \delta}$ can be bounded in $L^{\infty}$ norm by a continuous function of $h^\cC$.
\end{lemma}
\begin{proof}
Recall that the functions $M_n^+$ are built iteratively as solution of the equations
\[
\frac{\partial ( M_n \circ \phi^{-1}}{\partial \bar z} = \frac{J_{n+1}}{\partial_y \phi (\partial_1 + \Phi \partial_2) \Re( \phi )}
\]
where $J_{n+1}$ is a polynomial expression involving $\Phi$, derivatives of $\phi$ and $H$ and derivatives of $M_k$ for $k \le n$. Furthermore it is easy to see that all terms in $J_{n+1}$ involve at most derivatives up to order $n+2$ even when we consider any $M_k$ term already as a derivative of order $k$.

Since $\phi \in W^{34, p}$ for some $p > 2$, $\phi$ is in $\cC^{33}$ with $\alpha$-Hölder derivatives of order $33$ for some $\alpha$. Similarly $h^{\cC} \in W^{35, 3}$ and therefore $H$ is $\cC^{34}$ with Hölder derivatives. We see that $J_{2}$ is in $\cC^{33}$ and applying \cref{thm:cauchy} to all derivatives of $M_1 \circ \phi^{-1}$ we see that $M_1$ is in $\cC^{33}$. It is also clear by composition that $M_{1}$ is a continuous function of $h^\cC$.
We can then iterate the argument for all $M_n$ to obtain the result.
\end{proof}

\printbibliography

\end{document}